\documentclass[priprent]{elsarticle}
\usepackage{graphicx, epstopdf}
\usepackage{amssymb,amsmath}
\usepackage[toc,page,title,titletoc,header]{appendix}
\usepackage{amsbsy}
\usepackage{amsthm}
\usepackage{multirow}
\usepackage{indentfirst}
\usepackage{natbib}
\usepackage{amscd}
\numberwithin{equation}{section}

\newtheorem{Theorem}{Theorem}[section]

\newtheorem{Lemma}{Lemma}[section]

\newtheorem{Assumption-Notation}[Theorem]{Assumption-Notation}

\newtheorem{Proposition}{Proposition}[section]
\newtheorem{Remark}{Remark}[section]
\newtheorem{Corollary}{Corollary}[section]

\newtheorem{lem}{Lemma}

\newtheorem{rem}{Remark}

\long\def\symbolfootnote[#1]#2{\begingroup
\def\thefootnote{\fnsymbol{footnote}}\footnote[#1]{#2}\endgroup}

\begin{document}
\begin{frontmatter}
\title{The distribution of refracted L\'evy processes with jumps having rational Laplace transforms}

 \author{Jiang Zhou\corref{cor1}}
 \ead{1101110056@pku.edu.cn}
 \author{Lan Wu\corref{}}
 \ead{lwu@pku.edu.cn}
 \cortext[cor1]{Corresponding author.}

\address{School of Mathematical Sciences, Peking University, Beijing 100871, P.R.China}

\begin{abstract}
{We consider a refracted jump diffusion process having two-sided jumps with rational Laplace transforms.  For such a process, by applying a straightforward but interesting approach, we derive formulas for the Laplace transform of its distribution. Our formulas are presented in an attractive form and the approach is novel. In particular, the idea in the application of an approximating procedure is remarkable. Besides, the results are used to price Variable Annuities with state-dependent fees.
 }
\end{abstract}

\begin{keyword} Refracted L\'evy process; Rational Laplace transform; Wiener-Hopf factorization; Continuity theorem; Variable Annuities; State-dependent fees;
\end{keyword}
\end{frontmatter}

\section{Introduction}
A refracted L\'evy process $U=(U_t)_{t\geq 0}$ is derived from a L\'evy process $X=(X_t)_{t\geq 0}$  and is described by the following equation (see [13]):
\begin{equation}
U_t=X_t-\delta \int_{0}^{t}\textbf{1}_{\{U_s > b\}}ds,
\end{equation}
where $\delta, b \in \mathbb R$, and $\textbf{1}_{A}$ is the indicator function of a set $A$. There are several papers investigating refracted L\'evy processes, where the three papers [13,14,20] are based on the assumption that $X$ in (1.1) has negative jumps only; and in [23], the process $X$ is assumed to be a double-exponential jump diffusion process. Many results, including formulas for occupation times of $U$, have been obtained, and the interested reader is referred to the above papers for the details.
Besides, in [22,24,25], under several different assumptions on $X$, we have considered the following similar process $U^{s}=(U^{s}_t)_{t \geq 0}$:
\begin{equation}
dU^s_t = d X_t - \delta \textbf{1}_{\{U^s_t < b\}}dt.
\end{equation}
For the process $U$ in (1.1) with $X$ given by (2.1) (see below), we will show that $\mathbb P\left(U_t=b\right)=0$ for Lebesgue almost every $t > 0$ (see Remark 4.1), which means that $U_t=X_t- \delta t-(-\delta)\int_{0}^{t}\textbf{1}_{\{U_s < b\}}ds$ and thus the two processes $U^s$ and $U$ are equal essentially.

In this paper, we are interested in the distribution of $U$. When the process $X_t$ in (1.1) is a L\'evy process without positive jumps, the corresponding results can be found in [13]; see Theorem 6 (iv) in that paper. Thus here we focus on the situation that $X$ has both positive and negative jumps. In specific, we assume that $X$ in (1.1) is a jump diffusion process and its jumps have probability density functions whose Laplace transforms are rational functions. Such a L\'evy process is very popular and quite general, and two particular examples of it are
a hyper-exponential jump diffusion process (see [5,6]) and a L\'evy process with phase-type jumps (see [2,19]).
Under the above assumption on $X$, the purpose of this paper is to derive expressions for $\int_{0}^{\infty}e^{-q t} P\left(U_t < y\right)dt$ or in differential form
\begin{equation}
\int_{0}^{\infty}e^{-q t} P\left(U_t \in dy\right)dt,\ \ y \in \mathbb R,
\end{equation}
where $q > 0$. One reason why we are interested in the above quantity is that it is closely related to occupation times of $U$ since
\[q\int_{0}^{\infty}e^{-qt} E\left[\int_{0}^{t}\textbf{1}_{\{U_s<y\}}ds\right]dt=\int_{0}^{\infty}e^{-qt} P(U_t<y)dt, \ \ y \in \mathbb R,
\]
which can be derived by applying integration by parts. This means that the occupation times of $U$, i.e., $\int_{0}^{t}\textbf{1}_{\{U_s< y\}}ds$, can be derived from (1.3).

In [22], under the same assumption on $X$ as in this paper, we have derived formulas for $\int_{0}^{\infty}e^{- q t}P\left(U^s_t  < b\right)dt$, where $U^s$ and $b$ are given by (1.2). In this article, for given $b$ in (1.1), we combine the ideas in [22] with a novel and helpful approximating discussion to calculate $\int_{0}^{\infty}e^{- q t}P\left(U_t  < y\right)dt$, where $y \in \mathbb R$. Particularly, we obtain some attractive and uncommon formulas, which are written in terms of positive and negative Wiener-Hopf factors. These extraordinary expressions are important and are conjectured to hold for a general L\'evy process $X$ and the corresponding solution $U$ to (1.1), providing that such a solution $U$ exists.

Results in this paper have some applications. One application is to price  equity-linked investment products or Variable Annuities with state-dependent fees as in [24]. Such a state-dependent fee charging method is proposed recently and has several advantages (see [4,7]), e.g., it can reduce the incentive for a policyholder to surrender the policy. Equity-linked products are popular life insurance contracts and one reason for their popularity is that they typically provide a guaranteed minimum return. There are many papers studying Equity-linked products and their pricing, see, e.g., [10,15,18].
Investigations on evaluating Equity-linked products under a state-dependent fee structure are relatively new and the  reader is referred to [17] for a recent work.

The remainder of this paper is organized as follows. In Section 2, some notations and some preliminary  results are introduced. Next, we present an important proposition in Section 3 and give the main results in Section 4. Finally, the application of our main results is discussed in Section 5.

\section{Notations and preliminary  results}
In this paper, the process $X=(X_t)_{t \geq 0}$ in (1.1) is a jump diffusion process, where its jumps have rational Laplace transforms. Specifically,
\begin{equation}
  X_t = X_0 + \mu t+\sigma W_t + \sum_{k=1}^{N^+_t}Z^+_k-\sum_{k=1}^{N^-_t}Z^-_k,
\end{equation}
where $X_0$, $\mu$ and $\sigma > 0$ are constants; $(W_t)_{t\geq 0}$ is a standard Brownian motion; $\sum_{k=1}^{N^+_t}Z^+_k$ and
$\sum_{k=1}^{N^-_t}Z^-_k$ are compound Poisson processes with intensity $\lambda^+$ and $\lambda^-$, respectively; and the density functions of $Z_1^+$ and  $Z_1^-$ are given respectively by
\begin{equation}
p^+(z)=\sum_{k=1}^{m^+}\sum_{j=1}^{m_k}c_{kj}\frac{(\eta_k)^jz^{j-1}}{(j-1)!}e^{-\eta_k z}, \ \ z > 0,
\end{equation}
and
\begin{equation}
p^-(z)=\sum_{k=1}^{n^-}\sum_{j=1}^{n_k}d_{kj}\frac{(\vartheta_k)^jz^{j-1}}{(j-1)!}e^{-\vartheta_k z}, \ \ z > 0,
\end{equation}
with $\eta_i \neq \eta_j$ and $\vartheta_i \neq \vartheta_j$ for $i \neq j$;
moreover, $(W_t)_{t\geq 0}$, $\sum_{k=1}^{N^+_t}Z^+_k$ and $\sum_{k=1}^{N^-_t}Z^-_k$ are independent mutually.

\begin{Remark}
Parameters $\eta_k$ and $c_{kj}$ in (2.2) can take complex values as long as $p^+(z)$ satisfies $p^+(z)\geq 0$ and $\int_{0}^{\infty}p^+(z)dz=1$. In addition, if $\eta_1$ has the smallest real part  among $\eta_1$, $\ldots$, $\eta_{m^+}$, then
$
0< \eta_1< Re(\eta_2)\leq \cdots \leq Re(\eta_{m^+}).
$
\end{Remark}

\begin{Remark}
Formula (2.2) is quite general and particularly it contains phase-type distributions. Thus from Proposition $1$ in [2], we know that for any given L\'evy process $X$, there is a sequence of $X^n$ with the form of (2.1) such that
\[
\lim_{n \uparrow \infty} \sup_{s \in [0,t]}|X^n_s-X_s|=0, \ \ almost \ \ surely.
\]
\end{Remark}

In what follows, the law of $X$ starting from $x$ is denoted by $\mathbb P_x$ with $\mathbb E_x$ denoting  the corresponding expectation; when $x = 0$, we write $\mathbb P$ and $\mathbb E$ for convenience. And as usual, for $T \geq 0$, define
\begin{equation}
\underline{X}_{T}:=\inf_{0\leq t \leq T}X_t \ \ and   \ \ \overline{X}_{T}:=\sup_{0\leq t \leq T} X_t.
\end{equation}
Throughout this article, for given $q > 0$, $e(q)$ is an exponential random variable, whose expectation is equal to $\frac{1}{q}$. Besides, $e(q)$ is assumed to be independent of all stochastic processes appeared in the paper. In addition, for a complex value $x$, let $Re(x)$ and $Im(x)$ represent its real part and imaginary part, respectively.

For the L\'evy process $X$ given by (2.1),  it has been shown that equation (1.1) has a unique strong solution $U=(U_t)_{t\geq 0}$ (see, e.g., Theorem 305 of [21]),  which is a strong Markov process (see Remark 3 in [13]). For this unique solution $U$, our objective is deriving the expression of (1.3), i.e.,
\[
\int_{0}^{\infty}e^{-q t} \mathbb P_x\left(U_t \in dy\right)dt,\ \ y \in \mathbb R,
\]
and more importantly, we try to derive some novel expression.

Similar to previous investigations on refracted L\'evy processes (see, e.g., [13]), for given $\delta \in \mathbb R$, we introduce a process $Y$, which is defined as  $Y=\{Y_t=X_t-\delta t;t \geq 0\}$. For the process $Y$, the two quantities $\underline{Y}_T$ and $\overline{Y}_T$ are defined similarly to (2.4).  What is more, we denote by $\hat{\mathbb P}_y$ the law of $Y$ such that $Y_0=y$ and by $\hat{\mathbb E}_y$ the corresponding expectation, and write shortly $\hat{\mathbb P}$ and $\hat{\mathbb E}$ if $y=0$.

The following Lemma 2.1 gives the roots of $\psi(z)=q$ and $\hat{\psi}(z)= q$, where
\begin{equation}
\begin{split}
\psi(z):=
&-\frac{\sigma^2}{2}z^2 +  iz\mu +\lambda^+\left(\sum_{k=1}^{m^+}\sum_{j=1}^{m_k}c_{kj}\left(\frac{\eta_k}{\eta_k-iz}\right)^j-1\right)\\
&+\lambda^-\left(\sum_{k=1}^{n^-}\sum_{j=1}^{n_k}d_{kj}
\left(\frac{\vartheta_k}{\vartheta_k+iz}\right)^j-1\right),
\end{split}
\end{equation}
and $\hat{\psi}(z):= \psi(z) - i\delta z$.
Note that if $z \in \mathbb R$, then $\psi(z):= \ln\left(\mathbb E\left[e^{iz X_1}\right]\right)$ and $\hat{\psi}(z):= \ln\left(\hat{\mathbb E}\left[e^{iz Y_1}\right]\right)$.
Lemma 2.1 has been developed in [16]; see Lemma 1.1 and Theorem 2.1 in that paper (note that $\sigma > 0$ here).

\begin{Lemma}
(i) For $q > 0$, the equation $\psi(z)= q$ $(\hat{\psi}(z)=q)$ has, in the set $Im(z) < 0$, a total of $M^+ (\hat{M}^+)$ distinct solutions $-i\beta_{1}$ $(-i\hat{\beta}_{1})$, $-i\beta_{2}$ $ (-i\hat{\beta}_{2})$, $\ldots$, $-i\beta_{M^+}$ $(-i\hat{\beta}_{\hat{M}^+})$, with respective multiplicities $M_1=1 (\hat{M}_1=1)$, $M_{2}(\hat{M}_2)$, $\ldots$, $M_{M^+}(\hat{M}_{\hat{M}^+})$. Moreover,
\begin{equation}
0 < \beta_{1}< Re(\beta_{2})\leq \cdots \leq Re(\beta_{M^+}),\ \ 0 < \hat{\beta}_{1}< Re(\hat{\beta}_{2})\leq \cdots \leq Re(\hat{\beta}_{\hat{M}^+}),
\end{equation}
and
\begin{equation}
 \sum_{k=1}^{M^+}M_{k}=\sum_{k=1}^{\hat{M}^+}\hat{M}_{k}=1+\sum_{k=1}^{m^+}m_k.
\end{equation}

(ii)
For $q>0$ and $s \geq 0$,
\begin{equation}
\begin{split}
\hat{\mathbb E}\left[e^{-s \overline{Y}_{e(q)}}\right]
&=\prod_{k=1}^{m^+}\left(\frac{s+\eta_k}{\eta_k}\right)^{m_k}
\prod_{k=1}^{\hat{M}^+}\left(\frac{\hat{\beta}_{k}}{s+\hat{\beta}_{k}}\right)^{\hat{M}_k},
\end{split}
\end{equation}
and
\begin{equation}
\begin{split}
\mathbb E\left[e^{-s \overline{X}_{e(q)}}\right]
&=\prod_{k=1}^{m^+}\left(\frac{s+\eta_k}{\eta_k}\right)^{m_k}
\prod_{k=1}^{M^+}\left(\frac{\beta_{k}}{s+\beta_{k}}\right)^{M_{k}}.
\end{split}
\end{equation}
\end{Lemma}

Next, consider a function $F_1(x)$ on $(0,\infty)$ with the Laplace transform
\begin{equation}
\begin{split}
&\int_{0}^{\infty} e^{-s x}F_1(x)dx
=\frac{1}{s}\left(\frac{\hat{\mathbb E}\left[e^{-s \overline{Y}_{e(q)}}\right]}
{\mathbb E\left[e^{-s \overline{X}_{e(q)}}\right]}-1\right), \ \ s > 0.
\end{split}
\end{equation}
It follows from (2.8) and (2.9) that the Laplace transform of $F_1(x)$ is a rational function, which means that $F_1(x)$ is continuously differentiable on $(0,\infty)$. Since $\hat{M}_1=1$ and (2.6) holds, it can be shown that
\begin{equation}
\lim_{x \uparrow \infty}\frac{F_1(x)}{e^{-\hat{\beta}_1 x}}=-\prod_{k=2}^{\hat{M}^+}\left(\frac{\hat{\beta}_{k}}{\hat{\beta}_{k}-\hat{\beta}_1}\right)^{\hat{M}_k}
\prod_{k=1}^{M^+}\left(\frac{\beta_{k}-\hat{\beta}_1}{\beta_{k}}\right)^{M_{k}}.
\end{equation}
In addition, it holds that
\begin{equation}
\begin{split}
F_1(0):=\lim_{x \downarrow 0} F_1(x)=\lim_{s\uparrow \infty}\int_{0}^{\infty} s e^{-s x}F_1(x)dx
=\frac{\prod_{k=1}^{\hat{M}^+}\left(\hat{\beta}_{k}\right)^{\hat{M}_k}}
{\prod_{k=1}^{M^+}\left(\beta_{k}\right)^{M_{k}}}-1.
\end{split}
\end{equation}

\begin{Remark}
The expression of $F_1(x)$ can be obtained easily from (2.8)--(2.10) by using rational expansion, but it is not important in this paper and thus is omitted for brevity.
\end{Remark}

\begin{Remark}
Due to (2.11), $F_1(x)$ is absolutely integrable and the Laplace transform of $F_1(x)$ in (2.10) can be extended  analytically to the half-plane $Re(s) \geq 0$. When $s=0$, the right-hand side of (2.10) is understood as
\[
\lim_{s \downarrow 0}\frac{1}{s}\left(\frac{\hat{\mathbb E}\left[e^{-s \overline{Y}_{e(q)}}\right]}
{\mathbb E\left[e^{-s \overline{X}_{e(q)}}\right]}-1\right)=\frac{\partial}{\partial s}\left(
\prod_{k=1}^{\hat{M}^+}\left(\frac{\hat{\beta}_{k}}{s+\hat{\beta}_k}\right)^{\hat{M}_k}
\prod_{k=1}^{M^+}\left(\frac{s+\beta_{k}}{\beta_k}\right)^{M_{k}}\right)_{s=0}.
\]
Besides, $F_1(x)$ is bounded on $[0,\infty]$ with $F_1(\infty):=\lim_{x\uparrow \infty} F_1(x)=0$.
\end{Remark}

\begin{Lemma}
For the continuous function $F_1(x)$ given by (2.10), it holds that
\begin{equation}
F_1(x)+1 \geq 0, \ \ for \ \  x > 0.
\end{equation}
\end{Lemma}

\begin{proof}
For $q, s > 0$, we know (see, e.g., formula (4) in [1])
\begin{equation}
\mathbb E\left[e^{-s(\overline{X}_{e(q)}-h)}\textbf{1}_{\{\overline{X}_{e(q)}>h\}}\right]=
\mathbb E\left[e^{-q\tau_{h}^+-s(X_{\tau_h^+}-h)}\right]\mathbb E\left[e^{-sX_{e(q)}}\right], \ \ h > 0,
\end{equation}
where $\tau_h^+:=\inf\{t\geq 0: X_t>h\}$. Exchanging the order of integration yields

\[
\begin{split}
\int_{0}^{\infty}\mathbb E\left[e^{-s(\overline{X}_{e(q)}-h)}\textbf{1}_{\{\overline{X}_{e(q)}>h\}}\right]dh
&=\int_{0}^{\infty}
\int_{0}^{x}e^{-s(x-h)}dh\mathbb P\left(\overline{X}_{e(q)} \in dx\right)\\
&=\frac{1}{s}\left(1-\mathbb E\left[e^{-s \overline{X}_{e(q)}}\right]\right).
\end{split}
\]
Then, on both sides of (2.14), integrating with respect to $h$ from $0$ to $\infty$ gives
\begin{equation}
\int_{0}^{\infty} \mathbb E\left[e^{-q\tau_{h}^+-s(X_{\tau_h^+}-h)}\right]dh=\frac{1}{s}
\left(\frac{1}{\mathbb E\left[e^{-s \overline{X}_{e(q)}}\right]}-1\right).
\end{equation}
Formulas (2.10) and (2.15) mean that $\int_{0}^{\infty} e^{-s x}(F_1(x)+1)dx$ is a completely monotone function of $s$ on $(0,\infty)$, then (2.13) follows from Theorem 1a on page 439 in [8].
\end{proof}

The following Lemma 2.3 is taken from Proposition 1 (v) in [11], which states that for almost all $q>0$, $\psi(z) = q$ and $\hat{\psi}(z)= q$ only have simple solutions. Based on this lemma, we apply an
approximating argument (which reduces the calculation in a large extent) to derive the final results.

\begin{Lemma}
There exists  only finite numbers of $q > 0$ such that $\psi(z)= q$ or $\hat{\psi}(z)= q$ has solutions of  multiplicity greater than one.
\end{Lemma}

In the following, let $\mathbb S$ be the set of $q > 0$ such that all the roots of $\psi(z) = q$ and $\hat{\psi}(z) = q$ are simple.

\begin{Remark}
For $q \in \mathbb S$, Lemma 2.1 gives $M^+=\hat{M}^+=1+\sum_{k=1}^{m^+}m_k$.
\end{Remark}

We can obtain Lemma 2.4 by applying Lemma 2.1 to the dual processes $-X_t$ and $-Y_t$.

\begin{Lemma}
(i)  For $q \in \mathbb S$, the equation $\psi(z) = q$ $(\hat{\psi}(z) = q$) has, in the set $Im(z) > 0$, a total of $N^-$ $(\hat{N}^-)= \sum_{k=1}^{n^-}n_k+1$ distinct simple roots $i\gamma_{1}$
$(i\hat{\gamma}_{1}$),
$i\gamma_{2}$ $(i\hat{\gamma}_{2}$), $\ldots$, $i\gamma_{N^-}$ $(i\hat{\gamma}_{\hat{N}^-}$), ordered such that
\begin{equation}
0< \gamma_{1}< Re(\gamma_{2})\leq \cdots \leq Re(\gamma_{N^-}), \ \ 0< \hat{\gamma}_{1}< Re(\hat{\gamma}_{2})\leq \cdots \leq Re(\hat{\gamma}_{\hat{N}^-}).
\end{equation}
(ii) For $q \in \mathbb S$ and $Re(s) \geq 0$, we have
\begin{equation}
\begin{split}
&\mathbb E\left[e^{s\underline{X}_{e(q)}}\right]
=\prod_{k=1}^{n^-}\left(\frac{s+\vartheta_k}{\vartheta_k}\right)^{n_k}
\prod_{k=1}^{N^-}\left(\frac{\gamma_{k}}{s+\gamma_{k}}\right),\\
&\hat{\mathbb E}\left[e^{s\underline{Y}_{e(q)}}\right]
=\prod_{k=1}^{n^-}\left(\frac{s+\vartheta_k}{\vartheta_k}\right)^{n_k}
\prod_{k=1}^{\hat{N}^-}\left(\frac{\hat{\gamma}_{k}}{s+\hat{\gamma}_{k}}\right).
\end{split}
\end{equation}
\end{Lemma}

\begin{Remark}
For $q > 0$ and $q \in \mathbb S^c$, a similar result to (2.17) holds, e.g.,
\begin{equation}
\begin{split}
\hat{\mathbb E}\left[e^{s\underline{Y}_{e(q)}}\right]
&=\prod_{k=1}^{n^-}\left(\frac{s+\vartheta_k}{\vartheta_k}\right)^{n_k}
\prod_{k=1}^{\tilde{N}}\left(\frac{\hat{\gamma}_{k}}{s+\hat{\gamma}_{k}}\right)^{\hat{N}^-_k},
\end{split}
\end{equation}
where $\hat{N}^-_k$ is the multiplicity of $\hat{\gamma}_k$; and $\sum_{k=1}^{\tilde{N}}\hat{N}^-_k=\sum_{k=1}^{n^-}n_k+1$.
\end{Remark}

\begin{Remark}
For any $q > 0$, from (2.9), (2.17) and (2.18), it can be concluded that both $\overline{X}_{e(q)}$ and $\underline{Y}_{e(q)}$ have probability density functions.
\end{Remark}

The following Lemma 2.5 is important. We remark that the result in Lemma 2.5 is not surprising as $\sigma > 0$ in (2.1), and its proof is omitted since it can be established by using almost the same discussion as in Theorem 2.1 in [24].

\begin{Lemma}
For $q > 0$, the function $V_q(x)$, defined as $V_q(x):=\mathbb P_x\left(U_{e(q)}> y\right)$ for given $y > b$, is continuously differentiable on $\mathbb R$. Particularly, it holds that
\begin{equation}
V_q(b-)=V_q(b+)\ \ and \ \ V^{\prime}_q(b-)=V^{\prime}_q(b+).
\end{equation}
\end{Lemma}

The following Lemma 2.6 is known as the Wiener-Hopf factorization, we refer the reader to Theorem 6.16 in [12] for its derivation.

\begin{Lemma}
For a L\'evy process $X_t$, which is not a compound Poisson process, it holds that $X_{e(q)}-\underline{X}_{e(q)}$ is independent of $\underline{X}_{e(q)}$ and equal in distribution to
$\overline{X}_{e(q)}$. This means that
\[\mathbb E\left[e^{i \theta X_{e(q)}}\right]=\mathbb E\left[e^{i \theta \underline{X}_{e(q)}}\right]
\mathbb E\left[e^{i \theta \overline{X}_{e(q)}}\right],
\]
where $Re(\theta)=0$.
\end{Lemma}

To close this section, we present the following proposition, which gives the expression of
\[\mathbb P_x\left(U_{e(q)}>y\right)=q\int_{0}^{\infty}e^{-q t}\mathbb P_x\left(U_{t}>y\right)dt,\] where $y>b$ and $q \in \mathbb S$. The proof of Proposition 2.1 is long, thus it is left to the Appendix.
\begin{Proposition}
For $q \in \mathbb S$ and $y>b$, we have
\begin{equation}
\begin{split}
&\mathbb P_x\left(U_{e(q)}> y\right)=
\left\{\begin{array}{cc}
\sum_{k=1}^{M^+}J_{k}e^{\beta_{k}(x-b)},& x \leq b, \\
\sum_{k=1}^{\hat{M}^+}\hat{H}_{k}e^{\hat{\beta}_{k}(x-y)}+\sum_{k=1}^{\hat{N}^-}\hat{P}_{k}
e^{\hat{\gamma}_{k}(b-x)}, & b \leq x \leq y,\\
1+\sum_{k=1}^{\hat{N}^-}\hat{Q}_{k}e^{\hat{\gamma}_{k}(y-x)}+\sum_{k=1}^{\hat{N}^-}\hat{P}_{k}
e^{\hat{\gamma}_{k}(b-x)}, & x \geq y,
\end{array}\right.
\end{split}
\end{equation}
where $\hat{H}_k$ and $\hat{Q}_k$ are given by (A.17) and (A.18), respectively; $J_k$ and $\hat{P}_k$ are given by
rational expansion:
\begin{equation}
\begin{split}
&\sum_{i=1}^{M^+}\frac{J_i}{x-\beta_i}-
\sum_{i=1}^{\hat{N}^-}\frac{\hat{P}_i}{x+\hat{\gamma}_i}-\sum_{i=1}^{\hat{M}^+}\frac{\hat{H}_i}{x-\hat{\beta}_i}e^{\hat{\beta}_i(b-y)}\\
&=\frac{\prod_{k=1}^{m^+}(x-\eta_k)^{m_k}\prod_{k=1}^{n^-}(x+\vartheta_k)^{n_k}}
{\prod_{i=1}^{M^+}(x-\beta_i)\prod_{i=1}^{\hat{N}^-}(x+\hat{\gamma}_i)}\times\\
&\sum_{k=1}^{\hat{M}^+}
\frac{\prod_{i=1}^{M^+}(\hat{\beta}_k-\beta_i)\prod_{i=1}^{\hat{N}^-}(\hat{\beta}_k+\hat{\gamma}_i)}
{\prod_{i=1}^{m^+}(\hat{\beta}_k-\eta_i)^{m_i}\prod_{i=1}^{n^-}(\hat{\beta}_k+\vartheta_i)^{n_i}}
\frac{-\hat{H}_k}{x-\hat{\beta}_k}e^{\hat{\beta}_k(b-y)}.
\end{split}
\end{equation}
\end{Proposition}

\begin{Remark}
Formula (2.20) contains the roots of $\psi(z)=q$ and $\hat{\psi}(z)=q$, i.e., $\beta_k$, $\hat{\beta}_k$ and $\hat{\gamma}_k$. This poses a limitation to extend the result in Proposition 2.1 to a refracted L\'evy process $U$ driven by other L\'evy process, because we cannot characterize the roots of $\psi(z)=q$ for a general L\'evy process $X$ $($note that $\psi(z)=\ln\left(\mathbb E\left[e^{iz X_1}\right]\right)$ if $z \in \mathbb R$$)$.
In Theorem 4.1, we derive another expression for $\mathbb P_x\left(U_{e(q)} > y\right)$, which is free of $\beta_k$, $\hat{\beta}_k$ and $\hat{\gamma}_k$.
\end{Remark}

\section{An important result}
In this section, the following result is derived, and we have to say that the ideas in the derivation are interesting.
\begin{Proposition}
For given $y > b$, $q > 0$ and $Re(\phi)=0$, we have
\begin{equation}
\begin{split}
&\int_{-\infty}^{\infty}e^{-\phi(x-b)}\left(\mathbb P_x\left(U_{e(q)}> y\right) -\hat{\mathbb P}_x\left(Y_{e(q)}> y\right)\right)dx\\
&=\hat{\mathbb E}\left[e^{\phi \underline{Y}_{e(q)}}\right]\mathbb E\left[e^{\phi \overline{X}_{e(q)}}\right]
\int_{0}^{\infty}F_1(x+y-b)e^{\phi x}dx.
\end{split}
\end{equation}
\end{Proposition}
In the following, we first show that Proposition 3.1 holds for $q \in \mathbb S$ and then prove that it is also valid for $q \in \mathbb S^c$.

\subsection{Proof of Proposition 3.1 with $q \in \mathbb S$}
For $Re(\phi) =0$, it follows from (2.20) and (A.25) that
\begin{equation}
\begin{split}
&\int_{-\infty}^{\infty}e^{-\phi(x-b)}d\mathbb P_x\left(U_{e(q)}>y\right)
=\sum_{i=1}^{M^+}\frac{J_i\phi}{\beta_i-\phi}
+\sum_{i=1}^{\hat{N}^-}\frac{\hat{P}_i\phi}{\phi+\hat{\gamma}_i}
\\
&-\sum_{i=1}^{\hat{M}^+}\frac{\hat{H}_i\phi}{\hat{\beta}_i-\phi}e^{\hat{\beta}_i(b-y)}+e^{\phi(b-y)}
\left(\sum_{i=1}^{\hat{M}^+}\frac{\hat{H}_i \hat{\beta}_i}{\hat{\beta}_i-\phi}-\sum_{i=1}^{\hat{N}^-}\frac{\hat{Q}_i \hat{\gamma}_i}{\hat{\gamma}_i+\phi}\right).
\end{split}
\end{equation}

It can be proved that
\begin{equation}
\begin{split}
&\sum_{i=1}^{\hat{M}^+}\frac{\hat{H}_i \hat{\beta}_i}{\hat{\beta}_i-\phi}-\sum_{i=1}^{\hat{N}^-}\frac{\hat{Q}_i \hat{\gamma}_i}{\hat{\gamma}_i+\phi}
=1+\sum_{i=1}^{\hat{M}^+}\frac{\hat{H}_i \phi }{\hat{\beta}_i-\phi}+\sum_{i=1}^{\hat{N}^-}\frac{\hat{Q}_i \phi}{\hat{\gamma}_i+\phi}\\
&=\hat{\psi}^+(-\phi)\hat{\psi}^-(\phi)=\hat{\mathbb E}\left[e^{\phi \overline{Y}_{e(q)}}\right]\hat{\mathbb E}\left[e^{\phi \underline{Y}_{e(q)}}\right]=\hat{\mathbb E}\left[e^{\phi Y_{e(q)}}\right],
\end{split}
\end{equation}
where $\hat{\psi}^+(\cdot)$ and $\hat{\psi}^-(\cdot)$ are given by (A.5) and (A.6); the first equality is due to the fact that $\sum_{i=1}^{\hat{M}^+}\hat{H}_i-\sum_{i=1}^{\hat{N}^-}\hat{Q}_i-1=0$ (let $\theta \uparrow \infty$ in (A.29));
the second and the third equality follows respectively from (A.29) and (A.7), and the final one is a result of Lemma 2.6.

In addition, for $Re(\phi)=0$, applying integration by parts will give us
\begin{equation}
\begin{split}
&\int_{-\infty}^{\infty}e^{-\phi(x-b)}d\mathbb P_x\left(U_{e(q)}>y\right)
-e^{\phi (b-y)}\hat{\mathbb E}\left[e^{\phi Y_{e(q)}}\right]\\
&=\int_{-\infty}^{\infty}e^{-\phi(x-b)}d\mathbb P_x\left(U_{e(q)}>y\right)
-\int_{-\infty}^{\infty}e^{-\phi(x-b)}d\hat{\mathbb P}\left(Y_{e(q)}> y-x\right)\\
&=\phi \int_{-\infty}^{\infty}e^{-\phi(x-b)}\left(\mathbb P_x\left(U_{e(q)}>y\right)-\hat{\mathbb P}_x\left(Y_{e(q)}> y\right)\right)dx.
\end{split}
\end{equation}
Note that (see (3.10) in the following Lemma 3.1)
\[
\begin{split}
\int_{-\infty}^{\infty}|\mathbb P_x\left(U_{e(q)}>y\right)-\hat{\mathbb P}_x\left(Y_{e(q)}> y\right)|dx\leq \frac{|\delta|}{q}.
\end{split}
\]

Besides, from (2.21), (A.2), (A.4), (A.6), (A.7) and (A.17), after some straightforward calculations, we can derive (note that $M^+=\hat{M}^+$ if $q \in \mathbb S$)
\begin{equation}
\begin{split}
&\sum_{i=1}^{M^+}\frac{J_i\phi}{\beta_i-\phi}
+\sum_{i=1}^{\hat{N}^-}\frac{\hat{P}_i\phi}{\phi+\hat{\gamma}_i}
-\sum_{i=1}^{\hat{M}^+}\frac{\hat{H}_i\phi}{\hat{\beta}_i-\phi}e^{\hat{\beta}_i(b-y)}\\
&=\phi\hat{\mathbb E}\left[e^{\phi \underline{Y}_{e(q)}}\right]\mathbb E\left[e^{\phi \overline{X}_{e(q)}}\right]
\int_{0}^{\infty}F_0(x+y-b)e^{\phi x}dx,
\end{split}
\end{equation}
where
\begin{equation}
F_0(x)=\sum_{i=1}^{\hat{M}^+}e^{-\hat{\beta}_ix}\prod_{k=1}^{M^+}\frac{\hat{\beta}_i-\beta_k}{\beta_k}
\prod_{k=1,k \neq i}^{\hat{M}^+}\frac{\hat{\beta}_k}{\hat{\beta}_i-\hat{\beta}_k}, \ \ x > 0.
\end{equation}
Finally, for $s>0$, it holds that
\begin{equation}
\begin{split}
&\int_{0}^{\infty} e^{-s x}F_0(x)dx
=\sum_{i=1}^{\hat{M}^+}\prod_{k=1}^{M^+}\frac{\hat{\beta}_i-\beta_k}{\beta_k}
\prod_{k=1,k \neq i}^{\hat{M}^+}\frac{\hat{\beta}_k}{\hat{\beta}_i-\hat{\beta}_k}\frac{1}{\hat{\beta}_i+s}\\
&=\frac{1}{s}\left(\prod_{k=1}^{M^+}\frac{s+\beta_k}{\beta_k}
\prod_{k=1}^{\hat{M}^+}\frac{\hat{\beta}_k}{s+\hat{\beta}_k}
-1\right)
=\frac{1}{s}\left(\frac{\hat{\mathbb E}\left[e^{-s \overline{Y}_{e(q)}}\right]}
{\mathbb E\left[e^{-s \overline{X}_{e(q)}}\right]}-1\right),
\end{split}
\end{equation}
where the second equality follows from the rational expansion and the third one is due to (A.4), (A.5) and (A.7).

Formulas (2.10) and (3.7) conform that $F_1(x)=F_0(x)$ for $x > 0$.
Therefore, (3.1) for $q \in \mathbb S$ is derived from (3.2)--(3.5).

\subsection{Proof of Proposition 3.1 for $q>0$ and $q \in \mathbb S^c$}
First, for such a $q>0$ and $q \in \mathbb S^c$, Lemma 2.3 implies that there exists a sequences of $q_n \in \mathbb S$ such that $\lim_{n \uparrow \infty}q_n \downarrow q$. In Subsection 3.1, we have shown that (3.1) holds for $q_n$, this gives
\begin{equation}
\begin{split}
&\int_{-\infty}^{\infty}e^{-\phi(x-b)}\left(\mathbb P_x\left(U_{e(q_n)}> y\right) -\hat{\mathbb P}_x\left(Y_{e(q_n)}> y\right)\right)dx \\
&=\hat{\mathbb E}\left[e^{\phi \underline{Y}_{e(q_n)}}\right]\mathbb E\left[e^{\phi \overline{X}_{e(q_n)}}\right]
\int_{0}^{\infty}F_1^n(x+y-b)e^{\phi x}dx,
\end{split}
\end{equation}
where $Re(\phi)=0$ and
\begin{equation}
\begin{split}
&\int_{0}^{\infty} e^{-s x}F_1^n(x)dx
=\frac{1}{s}\left(\frac{\hat{\mathbb E}\left[e^{-s \overline{Y}_{e(q_n)}}\right]}
{\mathbb E\left[e^{-s \overline{X}_{e(q_n)}}\right]}-1\right), \ \  s > 0.
\end{split}
\end{equation}

\begin{Lemma}
For given $\delta, y \in \mathbb R$, we have
\begin{equation}
\begin{split}
&\int_{-\infty}^{\infty}|\mathbb P_x\left(U_t> y\right) -\hat{\mathbb P}_x\left(Y_t> y\right)|dx
\leq |\delta| t,
\end{split}
\end{equation}
and
\begin{equation}
\begin{split}
&\int_{-\infty}^{\infty}|\mathbb P_x\left(\overline{X}_t> y\right) -\hat{\mathbb P}_x\left(\overline{Y}_t> y\right)|dx
\leq |\delta| t.
\end{split}
\end{equation}
\end{Lemma}

\begin{proof}
Recall (1.1) and $Y_t=X_t-\delta t$.
For $\delta> 0$, it holds that
\[\mathbb P_x\left(X_t> y\right) \geq\mathbb P_x\left(U_t> y\right) \geq \hat{\mathbb P}_x\left(Y_t> y\right),\]
 thus
\[
\begin{split}
&\int_{-\infty}^{\infty}|\mathbb P_x\left(U_t> y\right) -\hat{\mathbb P}_x\left(Y_t> y\right)|dx
=\int_{-\infty}^{\infty}\mathbb P_x\left(U_t> y\right) -\hat{\mathbb P}_x\left(Y_t> y\right)dx \\
&\leq \int_{-\infty}^{\infty}\mathbb P_x\left(X_t> y\right) - \hat{\mathbb P}_x\left(Y_t> y\right)dx
=\int_{-\infty}^{\infty}\mathbb E\left[\textbf{1}_{\{y-x < X_t \leq y-x+\delta t\}}\right]dx\\
&=\mathbb E\left[\int_{-\infty}^{\infty}\textbf{1}_{\{y-x < X_t \leq y-x+\delta t\}}dx\right]=\delta t,
\end{split}
\]
where the penultimate equality is due to the Fubini theorem. For $\delta < 0$,  it is obvious that
 \[\mathbb P_x\left(X_t > y\right) \leq \mathbb P_x\left(U_t > y\right) \leq \hat{\mathbb P}_x\left(Y_t > y\right),\]
which gives
\[
\begin{split}
&\int_{-\infty}^{\infty}|\mathbb P_x\left(U_t> y\right) -\hat{\mathbb P}_x\left(Y_t> y\right)|dx
\leq \int_{-\infty}^{\infty}\hat{\mathbb P}_x\left(Y_t> y\right)-\mathbb P_x\left(X_t> y\right)dx\\
&=\mathbb E\left[\int_{-\infty}^{\infty}\textbf{1}_{\{y-x+\delta t < X_t \leq y-x\}}dx\right]=-\delta t,
\end{split}
\]
so (3.10) is derived.

Note  that
\[
\sup_{0\leq s \leq t}(X_s-\delta s) - |\delta| t \leq \overline{X}_t \leq \sup_{0\leq s \leq t}(X_s-\delta s) + |\delta| t,
\]
which means that
\[
\mathbb P_x\left(\overline{X}_t - |\delta| t > y\right) \leq \hat{\mathbb P}_x\left(\overline{Y}_t> y\right)\leq \mathbb P_x\left(\overline{X}_t + |\delta| t > y\right).
\]
Hence (3.11) can be proved similarly.
\end{proof}

\begin{Lemma}
For $Re(\phi)=0$,
\begin{equation}
\begin{split}
&\lim_{n \uparrow \infty}\int_{-\infty}^{\infty}e^{-\phi(x-b)}\left(\mathbb P_x\left(U_{e(q_n)}> y\right) -\hat{\mathbb P}_x\left(Y_{e(q_n)}> y\right)\right)dx\\
&=\int_{-\infty}^{\infty}e^{-\phi(x-b)}\left(\mathbb P_x\left(U_{e(q)}> y\right) - \hat{\mathbb P}_x\left(Y_{e(q)}> y\right)\right)dx,
\end{split}
\end{equation}
and for $Re(s)\geq 0$,
\begin{equation}
\begin{split}
\lim_{n \uparrow \infty}\hat{\mathbb E}\left[e^{s \underline{Y}_{e(q_n)}}\right]=\hat{\mathbb E}\left[e^{s \underline{Y}_{e(q)}}\right] \  and \  \lim_{n \uparrow \infty}\mathbb E\left[e^{-s \overline{X}_{e(q_n)}}\right]=\mathbb E\left[e^{-s \overline{X}_{e(q)}}\right].
\end{split}
\end{equation}
\end{Lemma}

\begin{proof}
Since $q_n > q$ and
$
\mathbb P_x\left(U_{e(q_n)} > y\right)=\int_{0}^{\infty}
q_n e^{-q_n t}\mathbb P_x\left(U_t > y\right)dt,
$
we derive via the dominated convergence theorem that
\begin{equation}
\begin{split}
&\lim_{n \uparrow \infty}\mathbb P_x\left(U_{e(q_n)} > y\right)=\mathbb P_x\left(U_{e(q)} > y\right).
\end{split}
\end{equation}
Similarly, we can prove (3.13) and the following result:
\begin{equation}
\begin{split}
&\lim_{n \uparrow \infty}\hat{\mathbb P}_x\left(Y_{e(q_n)} > y\right)
=\hat{\mathbb P}_x\left(Y_{e(q)} > y\right).
\end{split}
\end{equation}
In addition, it holds that
\begin{equation}
\begin{split}
&\int_{-\infty}^{\infty}e^{-\phi(x-b)}\left(\mathbb P_x\left(U_{e(q_n)}> y\right) - \hat{\mathbb P}_x\left(Y_{e(q_n)}> y\right)\right)dx\\
&= \int_{0}^{\infty} q_n e^{-q_n t}\int_{-\infty}^{\infty}e^{-\phi(x-b)}\left(\mathbb P_x\left(U_t> y\right) -\hat{\mathbb P}_x\left(Y_t> y\right)\right)dx dt.
\end{split}
\end{equation}
From (3.10) and (3.14)--(3.16), the dominated convergence theorem produces (3.12).
\end{proof}

\begin{Lemma}
For $F_1^n(x)$ in (3.9) and $F_1(x)$ in (2.10), it holds that
\begin{equation}
\lim_{n \uparrow \infty} \int_{0}^{\infty}e^{-s x} F_1^n(x)dx=\int_{0}^{\infty}e^{-s x} F_1(x)dx, \ \ Re(s) \geq 0.
\end{equation}
\end{Lemma}

\begin{proof}
Due to Remark 2.4, for each $n$, we have
\begin{equation}
\begin{split}
&\int_{0}^{\infty} e^{-s x}F_1^n(x)dx
=\frac{1}{s}\left(\frac{\hat{\mathbb E}\left[e^{-s\overline{Y}_{e(q_n)}}\right]}
{\mathbb E\left[e^{-s \overline{X}_{e(q_n)}}\right]}-1\right), \ \ for \ \ Re(s) \geq 0.
\end{split}
\end{equation}
Similar to the derivation of (3.13), it can be shown that
\begin{equation}
\begin{split}
\lim_{n \uparrow \infty}\hat{\mathbb E}\left[e^{-s \overline{Y}_{e(q_n)}}\right]=\hat{\mathbb E}\left[e^{-s \overline{Y}_{e(q)}}\right], \ \ Re(s) \geq 0.
\end{split}
\end{equation}
Formulas (3.13) and (3.19) lead to
\begin{equation}
\begin{split}
&\lim_{n \uparrow \infty} \int_{0}^{\infty} e^{-s x} F_1^n(x)dx= \int_{0}^{\infty} e^{-s x}F_1(x)dx
=\frac{1}{s}\left(\frac{\hat{\mathbb E}\left[e^{-s \overline{Y}_{e(q)}}\right]}
{\mathbb E\left[e^{-s \overline{X}_{e(q)}}\right]}-1\right),
\end{split}
\end{equation}
which holds for $Re(s) \geq 0$ and $s \neq 0$.

Next, we consider the case of $s=0$.  It follows from (3.9) that
\begin{equation}
\begin{split}
&\int_{0}^{\infty}F_1^n(x)dx
=\lim_{s \downarrow 0} \frac{1}{s}\left(\frac{\hat{\mathbb E}\left[e^{-s \overline{Y}_{e(q_n)}}\right]}
{\mathbb E\left[e^{-s \overline{X}_{e(q_n)}}\right]}-1\right)\\
&=\lim_{s \downarrow 0}\frac{\hat{\mathbb E}\left[e^{-s \overline{Y}_{e(q_n)}}\right]-
\mathbb E\left[e^{-s \overline{X}_{e(q_n)}}\right]}{s}\\
&=\lim_{s \downarrow 0}
\int_{0}^{\infty}e^{-sx}\left(\hat{\mathbb P}\left(\overline{Y}_{e(q_n)} \leq x\right)-\mathbb P\left(\overline{X}_{e(q_n)} \leq x\right)\right)dx\\
&=\int_{0}^{\infty}\left(\hat{\mathbb P}\left(\overline{Y}_{e(q_n)} \leq x\right)-\mathbb P\left(\overline{X}_{e(q_n)} \leq x\right)\right)dx,
\end{split}
\end{equation}
where the third equality is due to the integration by part and the final one follows from the dominated convergence theorem since (see (3.11))
\[
\int_{0}^{\infty}|\hat{\mathbb P}\left(\overline{Y}_{e(q_n)} \leq x\right)-\mathbb P\left(\overline{X}_{e(q_n)} \leq x\right)|dx \leq \int_{0}^{\infty}qe^{-qt} |\delta| t dt=\frac{|\delta|}{q}.
\]

As $q_n > q$ and (3.11) holds, formula (3.21) yields
\[
\begin{split}
&\lim_{n\uparrow \infty}\int_{0}^{\infty}F_1^n(x)dx=\lim_{n\uparrow \infty}\int_{0}^{\infty}q_n e^{-q_n t}\int_{0}^{\infty}\left(\hat{\mathbb P}\left(\overline{Y}_{t} \leq x\right)-\mathbb P\left(\overline{X}_{t} \leq x\right)\right)dx dt \\
&=\lim_{n\uparrow \infty}\int_{0}^{\infty}q e^{-q t}\int_{0}^{\infty}\left(\hat{\mathbb P}\left(\overline{Y}_{t} \leq x\right)-\mathbb P\left(\overline{X}_{t} \leq x\right)\right)dx dt=\int_{0}^{\infty}F_1(x)dx,
\end{split}
\]
which combined with (3.20), leads to (3.17).
\end{proof}

\begin{proof}[Proof of Proposition 3.1 for $q \in \mathbb S^c$]
Since $F_1(x), F_1^n(x) \geq -1$ for $x>0$ (see Lemma 2.2), we can define the following measures
\begin{equation}
M_1^n(x)=\int_{0}^{x}\Big(F_1^n(z)+1\Big)dz \ \ and \ \ M_1(x)=\int_{0}^{x}\Big(F_1(z)+1\Big)dz.
\end{equation}
Formula (3.17) implies that
\begin{equation}
\lim_{n \uparrow \infty} \int_{0}^{\infty} e^{-sx}dM_1^n(x)= \int_{0}^{\infty} e^{-sx}dM_1(x), \  \ s > 0,
\end{equation}
which combined with the continuity theorem for Laplace transform (see Theorem 2a on page 433 of [8]), gives
\begin{equation}
\lim_{n \uparrow \infty} \int_{0}^{x}\Big(F_1^n(z)+1\Big)dz= \int_{0}^{x}\Big(F_1(z)+1\Big)dz, \ \ for \ \ all \ \  x > 0.
\end{equation}

Next, for fixed $z>0$, introduce the following probability distribution functions
\begin{equation}
P_1^n(x)=\frac{\int_{0}^{x}\big(F_1^n(t)+1\big)dt}{\int_{0}^{z}\big(F_1^n(t)+1\big)dt} \ \ and \ \ P_1(x)=\frac{\int_{0}^{x}\big(F_1(t)+1\big)dt}{\int_{0}^{z}\big(F_1(t)+1\big)dt}, \ \ 0<x<z.
\end{equation}
Then, formula (3.24) means that $P_1^n(x)$ converges to $P_1(x)$ in distribution. As a result, $\lim_{n \uparrow \infty} \int_{0}^{z}e^{\phi t} dP_1^n(t)=
\int_{0}^{z}e^{\phi t} dP_1(t)$ for $Re(\phi)=0$, so
\begin{equation}
\lim_{n \uparrow \infty} \frac{\int_{0}^{z} e^{\phi t}\big(F_1^n(t)+1\big)dt}
{\int_{0}^{z} \big(F_1^n(t)+1\big)dt}= \frac{\int_{0}^{z} e^{\phi t} \big(F_1(t)+1\big)dt}
{\int_{0}^{z} \big(F_1(t)+1\big)dt}.
\end{equation}

It follows from (3.24) and (3.26) that
\begin{equation}
\lim_{n \uparrow \infty} \int_{0}^{x} e^{\phi t} F_1^n(t)dt=\int_{0}^{x} e^{\phi t} F_1(t)dt,  \ \ for \ \ any \ \ x > 0,
\end{equation}
which combined with (3.17), yields
\begin{equation}
\lim_{n \uparrow \infty} \int_{0}^{\infty} e^{\phi x} F_1^n(x+y-b)dx=\int_{0}^{\infty}
 e^{\phi x} F_1(x+y-b)dx.
\end{equation}
Therefore, the desired result that (3.1) holds also for $q \in \mathbb S^c$ follows from (3.8) by letting $n \uparrow \infty$ and using (3.12), (3.13) and (3.28).
\end{proof}

\section{Main results}
For the unique strong solution $U$ to (1.1) with $X$ given by (2.1), its probability distribution function is given by Theorems 4.1 and 4.2.

\begin{Theorem}
For $q>0$ and $y \geq b$,
\begin{equation}
\mathbb P_x\left(U_{e(q)}> y\right)=1-K_q(y-x)-\int_{b-x}^{y-x}F_1(y-x-z)K_q(dz),
\end{equation}
where  $K_q(x)$ is the convolution of $\underline{Y}_{e(q)}$ under $\hat{\mathbb P}$ and $\overline{X}_{e(q)}$ under $\mathbb P$, i.e.,
\begin{equation}
K_q(x)=\int_{-\infty}^{\min\{0,x\}}\mathbb P\left(\overline{X}_{e(q)} \leq x-z\right)\hat{\mathbb P}\left(\underline{Y}_{e(q)} \in dz \right), \ \ x \ \in \mathbb R,
\end{equation}
and $F_1(x)$ is continuous and differentiable on $(0,\infty)$ with rational Laplace transform given by (2.10).
\end{Theorem}

\begin{proof}
First, the right-hand side of (3.1) can be rewritten as
\begin{equation}
\begin{split}
&\int_{-\infty}^{\infty}e^{\phi x} \int_{-\infty}^{x}F_1(x-z+y-b)dK_q(z)dx,
\end{split}
\end{equation}
where $K_q(x)$ is given by (4.2). Since $\overline{X}_{e(q)}$ and $\underline{Y}_{e(q)}$ have  density functions (see Remark 2.7) and $F_1(x)$ is continuous on $(0,\infty)$, it is concluded that the integrand in (4.3), i.e., $\int_{-\infty}^{x}F_1(x-z+y-b)dK_q(z)$,  is continuous with respect to $x$.

As $Y_{e(q)}$ is the convolution of $\overline{Y}_{e(q)}$ and $\underline{Y}_{e(q)}$ (see Lemma 2.6), $\hat{\mathbb P}_x\left(Y_{e(q)} > y\right)$ is continuous with respect to $x$. This result and  Lemma 2.5 lead to that $\mathbb P_x\left(U_{e(q)}> y\right) -\hat{\mathbb P}_x\left(Y_{e(q)}> y\right)$ is also continuous. Therefore, for $y > b$, it follows from (3.1) and (4.3) that
\begin{equation}
\mathbb P_x\left(U_{e(q)}> y\right)-\hat{\mathbb P}_x\left(Y_{e(q)}> y\right)=\int_{-\infty}^{b-x}F_1(y-x-z)dK_q(z), \ \ x \in \mathbb R.
\end{equation}
In addition, for $Re(\phi)=0$, we have (recall (4.2) and Remark 2.4)
\begin{equation}
\begin{split}
\int_{-\infty}^{\infty}e^{\phi x}dK_q(x) \int_{0}^{\infty}e^{\phi x}F_1(x)dx
&=\frac{1}{\phi}\left(1-
\frac{\hat{\mathbb E}\left[e^{\phi \overline{Y}_{e(q)}}\right]}{\mathbb E\left[e^{\phi \overline{X}_{e(q)}}\right]}\right)\int_{-\infty}^{\infty}e^{\phi x}dK_q(x)\\
&=\frac{1}{\phi}\left(\int_{-\infty}^{\infty}e^{\phi x}dK_q(x)-\hat{\mathbb E}\left[e^{\phi Y_{e(q)}}\right]\right)\\
&=\int_{-\infty}^{\infty}e^{\phi x}\left(\hat{\mathbb P}\left(Y_{e(q)} \leq x\right)-K_q(x)\right)dx,
\end{split}
\end{equation}
where the second equality is due to Lemma 2.6 and the third one follows from the application of integration by parts. Note that
\[
\begin{split}
&\int_{-\infty}^{\infty}|\hat{\mathbb P}\left(Y_{e(q)} \leq x\right)-K_q(x)|dx\\
&\leq\int_{-\infty}^{\infty}\int_{-\infty}^{\min\{0,x\}}|\hat{\mathbb P}\left(\overline{Y}_{e(q)}\leq x-z\right)
-\mathbb P\left(\overline{X}_{e(q)}\leq x-z\right)|
\hat{\mathbb P}\left(\underline{Y}_{e(q)}\in dz\right)dx\\
&\leq\int_{-\infty}^{\infty}\int_{-\infty}^{0}|\hat{\mathbb P}\left(\overline{Y}_{e(q)}\leq x-z\right)
-\mathbb P\left(\overline{X}_{e(q)}\leq x-z\right)|
\hat{\mathbb P}\left(\underline{Y}_{e(q)}\in dz\right)dx\leq \frac{|\delta|}{q},
\end{split}
\]
where in the first inequality, we have used Lemma 2.6 and (4.2); the final inequality follows from (3.11).

For $x \in \mathbb R$, formula (4.5) gives
\begin{equation}
\begin{split}
\int_{-\infty}^{y-x}F_1(y-x-z)dK_q(z)
&=\hat{\mathbb P}\left(Y_{e(q)} \leq y - x\right)- K_q(y - x),
\end{split}
\end{equation}
which combined with (4.4), leads to
\begin{equation}
\begin{split}
&\mathbb P_x\left(U_{e(q)}> y\right)-\hat{\mathbb P}_x\left(Y_{e(q)}> y\right)\\
&=\hat{\mathbb P}\left(Y_{e(q)} \leq y - x\right)- K_q(y - x)- \int_{b-x}^{y-x}F_1(y-x-z)dK_q(z).
\end{split}
\end{equation}
This proves that (4.1) holds for $y > b $.
Letting $y \downarrow b$ in (4.7) and using that $\lim_{y \downarrow b}\int_{b-x}^{y-x}F_1(y-x-z)K_q(dz)=0$
(since $F_1(x)$ is bounded on $(0,\infty)$; see Remark 2.4) deduce that (4.1) holds also  for $y = b $.
\end{proof}

Similar derivations will lead to the following Theorem 4.2, and for the sake of brevity, the details are omitted.

\begin{Theorem}
 For $q>0$ and $y \leq  b$,
\begin{equation}
\mathbb P_x\left(U_{e(q)}< y\right)=K_q(y-x)-\int_{y-x}^{b-x}F_2(y-x-z)K_q(dz),
\end{equation}
where
$F_2(x)$ is continuous and differentiable on $(-\infty, 0)$ and its Laplace transform is given by
\begin{equation}
\begin{split}
&\int_{-\infty}^{0} e^{sz}F_2(z)dz
=\frac{1}{s}\left(\frac{\mathbb E\left[e^{s \underline{X}_{e(q)}}\right]}
{\hat{\mathbb E}\left[e^{s \underline{Y}_{e(q)}}\right]}-1\right), \ \ s > 0.
\end{split}
\end{equation}
\end{Theorem}

\begin{Remark}
It follows from (4.1) and (4.8) that
\[
\mathbb P_x\left(U_{e(q)} > b\right)+ \mathbb P_x\left(U_{e(q)} < b\right)=1,
\]
which implies that $\mathbb P_x\left(U_{e(q)} = b\right)=0$ for all $q>0$.
\end{Remark}

\begin{Remark}
For fixed $b \in \mathbb R$, letting $y=b$ in (4.1), we arrive at $
\mathbb P_x\left(U_{e(q)}> b\right)=1- K_q(b - x)
$, which means that
\[
\int_{-\infty}^{\infty}e^{-\phi(x-b)}d\left(\mathbb P_x\left(U_{e(q)} > b\right)\right)=\hat{\mathbb E}\left[e^{\phi\overline{Y}_{e(q)}}\right]\mathbb E\left[e^{\phi\underline{X}_{e(q)}}\right].
\]
A similar result has already been derived in [22]; see (4.9) in that paper.
\end{Remark}

\begin{Remark}
Compared with (2.20), in (4.1), (4.2), (4.8) and (4.9), the roots of $\psi(z)=q$ and $\hat{\psi}(z)=q$ disappear. The forms of these results and Remark 2.2 give us the following conjecture: formulas (4.1) and (4.8) hold for a general L\'evy process $X$ and the corresponding solution $U$ (if exists) to (1.1). Proving  this conjecture is a
potential research direction.
\end{Remark}

Since both $F_1(x)$ and $F_2(x)$ are differentiable, from (4.1) and (4.8), the expression of $\mathbb P_x\left(U_{e(q)} \in  dy\right)$ can be derived.

\begin{Corollary}
\begin{equation}
\begin{split}
&\mathbb P_x\left(U_{e(q)} \in  dy\right)=q\int_{0}^{\infty}e^{-q t}\mathbb P_x\left(U_{t} \in  dy\right)dt=\\
&\left\{\begin{array}{cc}
(F_1(0)+1)K_q(dy-x)+\int_{b-x}^{y-x}F_1^{\prime}(y-x-z)K_q(dz)dy,&y > b,\\
(F_2(0)+1)K_q(dy-x)- \int_{y-x}^{b-x}F_2^{\prime}(y-x-z)K_q(dz)dy, & y < b,
\end{array}\right.
\end{split}
\end{equation}
where $F_1(0)$ is given by (2.12), $F_2(0):=\lim_{x \uparrow 0}F_2(x)$ and moreover $F_2(0)=F_1(0)$.
\end{Corollary}

\begin{proof}
In (4.1) and (4.8), differentiating with respect to $y$ yields (4.10). Noting that $\mathbb P\left(U_{e(q)}=b\right)=0$ (see Remark 4.1), we can write $y > b$ or $y < b$ in (4.11) as  $y \geq b$ or $y \leq b$. An interesting conclusion is that $F_2(0)=F_1(0)$, which will be proved in the following.

For simplicity, we only consider $q \in \mathbb S$ since the case of $q \in \mathbb S^c$ can be shown  similarly.

It follows from (2.17) and (4.9) that
\begin{equation}
F_2(0)= \frac{\prod_{k=1}^{N^-}\gamma_k}{\prod_{k=1}^{\hat{N}^-}\hat{\gamma}_k}-1.
\end{equation}
For $q \in \mathbb S$, Lemma 2.1 (i) and Lemma 2.4 (i) give
\[
\hat{\psi}(z)- q = \psi(z)-i\delta z -q = -\frac{\sigma^2}{2}\frac{\prod_{k=1}^{\hat{M}^+}(\hat{\beta}_k-iz)\prod_{k=1}^{\hat{N}^-}(\hat{\gamma}_k+iz)}
{\prod_{k=1}^{n^-}(\vartheta_k+iz)^{n_k}\prod_{k=1}^{m^+}(\eta_k-iz)^{m_k}},
\]
which combined with the fact that $\psi(0)=0$, produces
\[
q = \frac{\sigma^2}{2}\frac{\prod_{k=1}^{\hat{M}^+}\hat{\beta}_k\prod_{k=1}^{\hat{N}^-}\hat{\gamma}_k}
{\prod_{k=1}^{n^-}(\vartheta_k)^{n_k}\prod_{k=1}^{m^+}(\eta_k)^{m_k}}.
\]
Similarly, we can prove
\[
q = \frac{\sigma^2}{2}\frac{\prod_{k=1}^{M^+}\beta_k\prod_{k=1}^{N^-}\gamma_k}
{\prod_{k=1}^{n^-}(\vartheta_k)^{n_k}\prod_{k=1}^{m^+}(\eta_k)^{m_k}}.
\]
Therefore,
\begin{equation}
\frac{\prod_{k=1}^{\hat{M}^+}\hat{\beta}_k}{\prod_{k=1}^{M^+}\beta_k}
=\frac{\prod_{k=1}^{N^-}\gamma_k}{\prod_{k=1}^{\hat{N}^-}\hat{\gamma}_k},
\end{equation}
and the desired result follows from (2.12), (4.11) and (4.12).
\end{proof}

\begin{Remark}
For a more general L\'evy process $X$, the two functions $F_1(x)$ and $F_2(x)$ given respectively  by (2.10) and (4.9)
may not be differentiable. Thus, it is better to understand (4.10) as
\[
\begin{split}
&\mathbb P_x\left(U_{e(q)} \in  dy\right)=\\
&\left\{\begin{array}{cc}
\big(F_1(0)+1\big)K_q(dy-x)+ \int_{b-x}^{y-x}F_1(dy-x-z)K_q(dz),&y > b,\\
\big(F_2(0)+1\big)K_q(dy-x)- \int_{y-x}^{b-x}F_2(dy-x-z)K_q(dz), & y < b.
\end{array}\right.
\end{split}
\]
\end{Remark}

\section{Applications in pricing Variable Annuities}
As stated in the introduction, our results can be used to price Variable Annuities (VAs) with state-dependent fees. First of all, we give some backgrounds.

VAs are life insurance products whose benefits are linked to the performance of a reference portfolio with guaranteed minimum returns. There are many kinds of guarantees such as Guaranteed Minimum Death Benefits (GMDBs) and Guaranteed Minimum Maturity Benefits (GMMBs), and the reader is referred to [3] for more details. Of course, the guaranteed benefits are not free. Traditionally, the corresponding fees are deducted at a fixed rate from the policyholder's account. This classical fee charging method has some disadvantages, which have been noted by [4]. Thus in [4], the authors proposed a new fee deducting approach under which only when the policyholder's account value is lower than a pre-specified level can the insurer charge fees. For more details and researches on this new method, we refer to [4,7,17,24,25].

Let $S_t$ and $F_t$ represent respectively the time-t value of the reference portfolio and the policyholder's account. Under the state-dependent fee structure, we have (see (1) in [4] or (2.3) in [24])
\begin{equation}
dF_t= F_{t-}\frac{d S_t}{S_{t-}}-(-\delta) F_{t-} \textbf{1}_{\{F_{t-} < B\}}dt ,\ \ t>0,
\end{equation}
where $-\delta >0$ is the fee rate and $B$ is a pre-specified level. Note that the case of $B=\infty$ corresponds to the classical fee charging method. Furthermore, assume that
\[S_t=S_0e^{X_t-\delta t},\]
with $X_t$ given by (2.1).

For a VA with GMMBs, its payoff can be written as $G(F_{T})$, where $T$ is the maturity and $G(\cdot)$ is a payoff function. For a VA with GMDBs, its payment when the policyholder dies is given by $G(F_{T_x})$, where $T_x$ is the time of the death of the insured. A simple example of $G(\cdot)$ is $G(x)=\max\{x,K\}$, where $K$ is a constant. In order to price VAs with GMMBs or GMDBs, we need to compute the following expectations under an equivalent martingale measure:
\begin{equation}
\mathbb E\left[e^{- rT}G(F_{T})\right] \ \ or \ \ \mathbb E\left[e^{- rT_x}G(F_{T_x})\right],
\end{equation}
where $r>0$ denotes the continuously compounded constant risk-free rate.

As the market is incomplete, an equivalent martingale measure should be chosen to
calculate (5.2). Similar to [24], we use the Cram\'er-Esscher transform (see [9]) to obtain the wanted martingale measure. In specific, define first
\[
\frac{d \mathbb P^c}{\mathbb P}=\frac{e^{c X_t}}{\mathbb E\left[e^{c X_t}\right]},
\]
where $c \in \mathbb R$ such that $\mathbb E\left[e^{c X_t}\right]< \infty$. And for convenience, in (2.2) and (2.3), we assume that $\eta_1$ ($\vartheta_1$) has the smallest real part
among $\eta_1$, $\ldots$, $\eta_{m^+}$ ($\vartheta_1$, $\ldots$, $\vartheta_{n^-}$). As $S_t=S_0e^{X_t-\delta t}$, it is reasonable to require that $\mathbb E\left[e^{X_t}\right]<\infty$, this means that $\eta_1>1$ in (2.2). Note that $\lim_{c\uparrow \eta_1}\mathbb E\left[e^{c X_t}\right]=\infty$ and $\lim_{c\downarrow -\vartheta_1}\mathbb E\left[e^{c X_t}\right]=\infty$. We can choose $c^{*}$ such that $e^{-r t}S_t$ is a martingale under $\mathbb P^{c^*}$. It is obvious that $X_t$ is still a L\'evy process under $\mathbb P^{c^*}$, and in particular, the process $X$ has the same form as (2.1) under $\mathbb P^{c^*}$. So we drop the superscript $c^*$ from $\mathbb P^{c^*}$ and assume that the expectations appeared in the following are calculated under the equivalent martingale measure $\mathbb P^{c^*}$.

For a VA with GMDBs, its price is $\mathbb E\left[e^{- rT_x}G(F_{T_x})\right]$. Applying similar discussions presented in [24] (see the derivation of (2.9) in that paper), we obtain that the computation of $\mathbb E\left[e^{- rT_x}G(F_{T_x})\right]$ reduces to calculate $\mathbb E\left[e^{- r e(q)}G(F_{e(q)})\right]=\frac{q}{r+q}\mathbb E\left[G(F_{e(r+q)})\right]$ for given $q>0$. For a VA with GMDBs, we note that
\[
\int_{0}^{\infty}e^{-sT} \mathbb E\left[e^{- rT}G(F_{T})\right]dT=\frac{1}{s+r}\mathbb E\left[G(F_{e(s+r)})\right],
\]
from which $\mathbb E\left[e^{- rT}G(F_{T})\right]$ can be obtained by using a numerical Laplace inversion technique.

In summary, the key step to price a VA with GMDBs or GMMBs is deriving the expression of $\mathbb E\left[G(F_{e(q)})\right]$ for $q>0$.

From (5.1), applying It$\hat{o}$'s formula gives
$F_t = F_0e^{U_t}$ with
\begin{equation}
dU_t = d(X_t-\delta t) - (-\delta) \textbf{1}_{\{U_t < b\}}dt=dX_t-\delta \textbf{1}_{\{U_t > b\}}dt,
\end{equation}
where $b = \ln\left(\frac{B}{F_0}\right)$. So we arrive at
\begin{equation}
\mathbb E\left[G(F_{e(q)})\right]=\mathbb E\left[G(F_{0}e^{U_{e(q)}})\right]=\int_{-\infty}^{\infty}G(F_0e^y) \mathbb P\left(U_{e(q)} \in dy \right).
\end{equation}
It follows from (1.1), (4.10) and (5.3) that
\begin{equation}
\begin{split}
&\mathbb P\left(U_{e(q)} \in  dy\right)=
\left\{\begin{array}{cc}
(F_1(0)+1)K_q(dy)+\int_{b}^{y}F_1^{\prime}(y-z)K_q(dz)dy,&y > b,\\
(F_2(0)+1)K_q(dy)- \int_{y}^{b}F_2^{\prime}(y-z)K_q(dz)dy, & y < b,
\end{array}\right.
\end{split}
\end{equation}
where $F_1(x)$, $F_2(x)$ and $K_q(x)$ are given respectively by (2.10), (4.9) and (4.2).

By applying rational expansion, we can obtain semi-explicit expressions for $F_1(x)$, $F_2(x)$ and $K_q(x)$ and thus for $\mathbb P\left(U_{e(q)} \in  dy\right)$. However, for $q \in \mathbb S^c$, formulas for $\mathbb P\left(U_{e(q)} \in  dy\right)$ are very long and complicated, and more importantly, they are difficult to be used in numerical computations since we need to handle multiple roots. Fortunately, due to Lemma 2.3, it is safe and convenient to consider only the case of $q \in \mathbb S$. The corresponding results will be given in the following corollary, from which we can obtain first the expression of $\mathbb E\left[G(F_{e(q)})\right]$ and then the price of a VA with GMDBs or GMMBs.

\begin{Corollary}
For $q \in \mathbb S$, defining $f_q(x):=\mathbb P\left(U_{e(q)} \in  dy\right)/dy$, we have the following results.

(i) If $b\geq 0$, then
\[
f_q(y)=
\left\{
\begin{split}
&\frac{\prod_{k=1}^{\hat{M}^+}\hat{\beta}_{k}}
{\prod_{k=1}^{M^+}\beta_{k}}K_q(dy)+\sum_{i=1}^{M^+}\sum_{j=1}^{\hat{N}^-}\sum_{m=1}^{\hat{M}^+}
\frac{K_{i,j}F_{1,m}}{\hat{\beta}_m-\beta_i}\left(e^{-\beta_i y}-e^{-\beta_i b }e^{\hat{\beta}_m (b-y)}\right), \ \ y> b,\\
&\frac{\prod_{k=1}^{N^-}\gamma_k}{\prod_{k=1}^{\hat{N}^-}\hat{\gamma}_k}K_q(dy)-
\sum_{i=1}^{M^+}\sum_{j=1}^{\hat{N}^-}\sum_{n=1}^{N^-}\frac{K_{i,j}F_{2,n}}{\gamma_n+\beta_i}
\left(e^{-\beta_i y}-e^{-\beta_i b }e^{\gamma_n (y-b)}\right), \ \ 0<y<b,
\end{split}
\right.
\]
and for $y\leq 0$,
\begin{small}
\[
f_q(y)=\frac{\prod_{k=1}^{N^-}\gamma_k}{\prod_{k=1}^{\hat{N}^-}\hat{\gamma}_k}K_q(dy)-
\sum_{i=1}^{M^+}\sum_{j=1}^{\hat{N}^-}\sum_{n=1}^{N^-}
\frac{K_{i,j}F_{2,n}}{\hat{\gamma}_j-\gamma_n}
\left(e^{\gamma_n y}\left(\frac{\hat{\gamma}_j+\beta_i}{\gamma_n+\beta_i}-\frac{\hat{\gamma}_j-\gamma_n}
{\gamma_n+\beta_i}e^{-(\beta_i+\gamma_n) b }\right)-e^{\hat{\gamma}_j y}\right).
\]
\end{small}

(ii) If $b<0$, then
\[
f_q(y)=
\left\{
\begin{split}
&\frac{\prod_{k=1}^{N^-}\gamma_k}{\prod_{k=1}^{\hat{N}^-}\hat{\gamma}_k}K_q(dy)-
\sum_{i=1}^{M^+}\sum_{j=1}^{\hat{N}^-}K_{i,j}\sum_{n=1}^{N^-}F_{2,n}
\frac{e^{\gamma_n(y-b)+\hat{\gamma}_j b}-e^{\hat{\gamma}_j y}}{\hat{\gamma}_j-\gamma_n}, \ \ y< b,\\
&\frac{\prod_{k=1}^{\hat{M}^+}\hat{\beta}_{k}}
{\prod_{k=1}^{M^+}\beta_{k}}K_q(dy)+
\sum_{i=1}^{M^+}\sum_{j=1}^{\hat{N}^-}K_{i,j}\sum_{m=1}^{\hat{M}^+}F_{1,m}
\frac{e^{\hat{\gamma}_j y}-e^{\hat{\beta}_m(b-y)+\hat{\gamma}_j b}}{\hat{\gamma}_j+\hat{\beta}_m}, \ \ b<y<0,
\end{split}
\right.
\]
and for $y\geq 0$,
\begin{small}
\[
f_q(y)=\frac{\prod_{k=1}^{\hat{M}^+}\hat{\beta}_{k}}
{\prod_{k=1}^{M^+}\beta_{k}}K_q(dy)+\sum_{i=1}^{M^+}\sum_{j=1}^{\hat{N}^-}
\sum_{m=1}^{\hat{M}^+}\frac{K_{i,j}F_{1,m}}{\beta_i-\hat{\beta}_m}
\left\{e^{-\beta_m y}\left(\frac{\beta_i+\hat{\gamma}_j}{\hat{\gamma}_j+\hat{\beta}_m}+\frac{\hat{\beta}_m-\beta_i}
{\hat{\beta}_m+\hat{\gamma}_j}e^{(\hat{\beta}_m+\hat{\gamma}_j)b}\right)
-e^{-\beta_i y}\right\}.
\]
\end{small}

In the above formulas,
\begin{equation}
F_{1,i}=-\hat{\beta}_i \prod_{k=1}^{M^+}\frac{\hat{\beta}_i-\beta_k}{\beta_k}
\prod_{k=1,k \neq i}^{\hat{M}^+}\frac{\hat{\beta}_k}{\hat{\beta}_i-\hat{\beta}_k}, \ \ for \ \ 1\leq i \leq \hat{M}^+,
\end{equation}

\begin{equation}
F_{2,i}=-\gamma_i \prod_{k=1}^{\hat{N}^-}\frac{\hat{\gamma}_k-\gamma_i}{\hat{\gamma}_k}
\prod_{k=1,k \neq i}^{N^-}\frac{\gamma_k}{\gamma_k-\gamma_i}, \ \ for \ \ 1\leq i\leq N^-,
\end{equation}
and
\begin{equation}
K_q(dx)=\sum_{i=1}^{M^+}\sum_{j=1}^{\hat{N}^-}K_{i,j}
e^{-\beta_i x}e^{(\beta_i+\hat{\gamma}_j) (x \wedge 0)},
\end{equation}
where
\[
K_{i,j}=\frac{\beta_i\hat{\gamma}_j}{\beta_i+\hat{\gamma}_j}\prod_{k=1}^{m^+}\frac{(\eta_k-\beta_{i})^{m_k}}{(\eta_k)^{m_k}}
\prod_{k=1,k \neq i}^{M^+}\frac{\beta_{k}}{\beta_{k}-\beta_{i}}\prod_{k=1}^{n^-}\frac{(\vartheta_k-\hat{\gamma}_{j})^{n_k}}
{(\vartheta_k)^{n_k}}
\prod_{k=1,k \neq j}^{\hat{N}^-}\left(\frac{\hat{\gamma}_{k}}{\hat{\gamma}_{k}-\hat{\gamma}_{j}}\right).
\]

\end{Corollary}

\begin{proof}
Since $q \in \mathbb S$, we have $F_1(x)=F_0(x)$ (see (3.6) and (3.7)). So for $x>0$,
\[
F_1^{\prime}(x)=\sum_{i=1}^{\hat{M}^+}F_{1,i}e^{-\hat{\beta}_i x}
\]
with $F_{1,i}$ given by (5.6). Besides, formula (2.12) gives
\[
\begin{split}
F_1(0)=\frac{\prod_{k=1}^{\hat{M}^+}\hat{\beta}_{k}}
{\prod_{k=1}^{M^+}\beta_{k}}-1.
\end{split}
\]

From (2.17) and (4.9), applying partial fraction expansion gives
\[
F_2^{\prime}(x)=\sum_{i=1}^{N^-}F_{2,i}e^{\gamma_i x}, \ \ x < 0,
\]
where $F_{2,i}$ is given by (5.7). In addition, we know (see (4.11))
\[
F_2(0)= \frac{\prod_{k=1}^{N^-}\gamma_k}{\prod_{k=1}^{\hat{N}^-}\hat{\gamma}_k}-1.
\]

From Lemma A.1 (i), Lemma A.1 (iii) and (4.2), some straightforward calculations leads to (5.8).

Therefore, the desired results follow from (5.5) after some simple computations.
\end{proof}

\bigskip
\begin{appendix}
\medskip

\renewcommand{\therem}{A.\arabic{rem}}
\renewcommand{\thelem}{A.\arabic{lem}}
\renewcommand{\thecor}{A.\arabic{cor}}

\section*{Appendix}
The proof of Proposition 2.1 is given in this section, where some ideas used can also be found in [22]. For completeness and for the convenience of the reader, we present all the details rather than omit some of them even though we will repeat some
preliminary results and calculation procedures appeared in [22].

Recall $\mathbb S$ is the set of $q > 0$ such that all the roots of $\psi(z) = q$ and $\hat{\psi}(z) = q$ are simple.

The following Lemma A.1 follows directly from Lemmas 2.1 and 2.4.

\begin{lem}
 For $q \in \mathbb S$, the following results hold.

(i) For $ y \geq 0$, $
\mathbb P\left(\overline{X}_{e(q)}\in dy\right)=\sum_{k=1}^{M^+}C_{k}e^{-\beta_{k}y}dy$,
where
\begin{equation}
\frac{C_i}{\beta_{i}}=\prod_{k=1}^{m^+}\left(\frac{\eta_k-\beta_{i}}{\eta_k}\right)^{m_k}
\prod_{k=1,k \neq i}^{M^+}\frac{\beta_{k}}{\beta_{k}-\beta_{i}}, \ \  1\leq i \leq M^+.\tag{A.1}
\end{equation}

(ii) For $y \geq 0$, $\hat{\mathbb P}\left(\overline{Y}_{e(q)}\in dy\right)=\sum_{k=1}^{\hat{M}^+}\hat{C}_{k}e^{-\hat{\beta}_{k}y}dy$, where
\begin{equation}
\frac{\hat{C}_i}{\hat{\beta}_{i}}=\prod_{k=1}^{m^+}\left(\frac{\eta_k-\hat{\beta}_{i}}{\eta_k}\right)^{m_k}
\prod_{k=1,k \neq i}^{\hat{M}^+}\frac{\hat{\beta}_{k}}{\hat{\beta}_{k}-\hat{\beta}_{i}}, \ \  1\leq i \leq \hat{M}^+.\tag{A.2}
\end{equation}

(iii) For $y \leq 0$, $
\hat{\mathbb P}\left(\underline{Y}_{e(q)}\in dy\right)=\sum_{k=1}^{\hat{N}^-}
\hat{D}_{k}e^{\hat{\gamma}_{k}y}dy
$, where
\begin{equation} \frac{\hat{D}_j}{\hat{\gamma}_{j}}=\prod_{k=1}^{n^-}\left(\frac{\vartheta_k-\hat{\gamma}_{j}}
{\vartheta_k}\right)^{n_k}
\prod_{k=1,k \neq j}^{\hat{N}^-}\left(\frac{\hat{\gamma}_{k}}{\hat{\gamma}_{k}-\hat{\gamma}_{j}}\right), \ \ 1\leq j \leq \hat{N}^-.\tag{A.3}
\end{equation}
\end{lem}

Next, introduce the following three rational functions:
\begin{equation}
\psi^+(s):=\prod_{k=1}^{m^+}\left(\frac{s+\eta_k}{\eta_k}\right)^{m_k}
\prod_{k=1}^{M^+}\left(\frac{\beta_{k}}{s+\beta_{k}}\right)
=\sum_{k=1}^{M^+}\frac{C_k}{s+\beta_k}, \tag{A.4}
\end{equation}
\begin{equation}
\hat{\psi}^+(s):=\prod_{k=1}^{m^+}\left(\frac{s+\eta_k}{\eta_k}\right)^{m_k}
\prod_{k=1}^{\hat{M}^+}\left(\frac{\hat{\beta}_{k}}{s+\hat{\beta}_{k}}\right)
=\sum_{k=1}^{\hat{M}^+}\frac{\hat{C}_k}{s+\hat{\beta}_k},
\tag{A.5}
\end{equation}
and
\begin{equation}
\begin{split}
&\hat{\psi}^-(s):=\prod_{k=1}^{n^-}\left(\frac{s+\vartheta_k}{\vartheta_k}\right)^{n_k}
\prod_{k=1}^{\hat{N}^-}\left(\frac{\hat{\gamma}_{k}}{s+\hat{\gamma}_{k}}\right)
=\sum_{k=1}^{\hat{N}^-}\frac{\hat{D}_{k}}{s+\hat{\gamma}_{k}}.
\end{split}\tag{A.6}
\end{equation}
For $q \in \mathbb S$ and $Re(s) \geq 0$, note that (see (2.8), (2.9) and (2.17))
\begin{equation}
\begin{split}
\mathbb E\left[e^{-s \overline{X}_{e(q)}}\right]=\psi^+(s), \ \ \hat{\mathbb E}\left[e^{-s \overline{Y}_{e(q)}}\right]=\hat{\psi}^+(s) \  \ and \ \
\hat{\mathbb E}\left[e^{s\underline{Y}_{e(q)}}\right]=\hat{\psi}^-(s).
\end{split}\tag{A.7}
\end{equation}
In addition, for $a \in \mathbb R$, define
\begin{equation}
\tau_{a}^{+} := \inf\{t\geq 0: X_t > a\}\ \ and \ \  \hat{\tau}_{a}^{-}:= \inf\{t\geq 0: Y_t < a \}.\tag{A.8}
\end{equation}

Results on the one-sided exit problems of $X$ and $Y$ are presented in the following lemma. Lemma A.2 (i) can be established by applying Lemma A.1 (i), (2.14) and (A.4); and Lemma A.2 (ii) follows from Lemma A.1 (iii), (A.6) and  the following result (see Corollary 2 in [1])
\[
\hat{\mathbb E}\left[e^{-q \hat{\tau}_{x}^-+s(Y_{\hat{\tau}_{x}^-}-x)}\right]=\frac{\hat{\mathbb E}\left[
e^{s(\underline{Y}_{e(q)}-x)}\textbf{1}_{\{\underline{Y}_{e(q)} < x \}}\right]}{\hat{\mathbb E}\left[
e^{s\underline{Y}_{e(q)}}\right]}, \ \ x, s \geq 0.
\]

\begin{lem}
(i) For $q \in \mathbb S$ and $x, y \geq  0$,
\[
\mathbb E\left[e^{-q \tau_{x}^+}\rm{\bf{1}}_{\{X_{\tau_{x}^+}- x \in dy\}}\right]=
C_0(x)\delta_0(dy)+
\sum_{k=1}^{m^+}\sum_{j=1}^{m_k}C_{kj}(x)\frac{(\eta_k)^jy^{j-1}}{(j-1)!}e^{-\eta_k y}dy,
\]
where $\delta_0(dy)$ is the Dirac delta at $y = 0$, $C_0(x)$ and $C_{kj}(x)$ are given by rational expansion:
\begin{equation}
\begin{split}
&C_0(x)+\sum_{k=1}^{m^+}\sum_{j=1}^{m_k}C_{kj}(x)\left(\frac{\eta_k}{\eta_k+ s}\right)^j
=\frac{1}{\psi^+(s)}\sum_{k=1}^{M^+}C_{k}\frac{e^{-\beta_k x}}{s+\beta_k}, \ \ x\geq 0.
\end{split}\tag{A.9}
\end{equation}

(ii) For $q \in  \mathbb S$ and $x, y \leq 0$,
\[
\hat{\mathbb E}\left[e^{-q \hat{\tau}_{x}^-}\rm{\bf{1}}_{\{Y_{\hat{\tau}_{x}^-}- x \in dy\}}\right]=\hat{D}_0(x)\delta_0(dy)+
\sum_{k=1}^{n^-}\sum_{j=1}^{n_k}\hat{D}_{kj}(x)\frac{(\vartheta_k)^j(-y)^{j-1}}{(j-1)!}e^{\vartheta_k y}dy,
\]
where $\hat{D}_0(x)$ and $\hat{D}_{kj}(x)$ are given by rational expansion:
\begin{equation}
\begin{split}
&\hat{D}_0(x)+\sum_{k=1}^{n^-}\sum_{j=1}^{n_k}\hat{D}_{kj}(x)\left(\frac{\vartheta_k}{\vartheta_k+ s}\right)^j
=
\frac{1}{\hat{\psi}^-(s)}\sum_{k=1}^{\hat{N}^-}\hat{D}_{k}\frac{e^{\hat{\gamma}_k x}}{s+\hat{\gamma}_k},  \  x\leq 0.
\end{split}\tag{A.10}
\end{equation}
\end{lem}

\begin{rem}
A useful observation is that $C_0(x)$ and $C_{kj}(x)$ in (A.9) are linear combinations of $e^{\beta_i x}$ for  $1 \leq i \leq M^+$, and $\hat{D}_0(x)$ and $\hat{D}_{kj}(x)$ in (A.10) are linear combinations of $e^{\hat{\gamma}_i x}$ for $1\leq i \leq \hat{N}^-$.
\end{rem}

Lemma A.3 is a straightforward result of (A.9) and (A.10), here the reader is reminded that $\frac{1}{(\theta +\beta_k)
(s+\beta_k)}$ can be written as $\frac{1}{s-\theta}\left(\frac{1}{\theta +\beta_k}-\frac{1}
{s+\beta_k}\right)$.

\begin{lem} For any $\theta >0$ and $s \neq -\eta_1, \ldots, -\eta_{m^+}$ with $\theta \neq s$,
\begin{equation}
\begin{split}
&\int_{0}^{\infty}e^{-\theta x}C_0(x)dx+\sum_{k=1}^{m^+}\sum_{j=1}^{m_k} \int_{0}^{\infty}e^{-\theta x}C_{kj}(x)dx\left(\frac{\eta_k}{\eta_k+s}\right)^j=\frac{1}{s-\theta}\left(\frac{\psi^+(\theta)}
{\psi^+(s)}-1\right),
\end{split}\tag{A.11}
\end{equation}
and for any $\theta >0$ and $s \neq -\vartheta_1, \ldots, -\vartheta_{n^-}$ with $\theta \neq s$
\begin{equation}
\begin{split}
&\int_{-\infty}^{0}e^{\theta x} \hat{D}_0(x)dx +\sum_{k=1}^{n^-}\sum_{j=1}^{n_k} \int_{-\infty}^{0}e^{\theta x}\hat{D}_{kj}(x)dx\left(\frac{\vartheta_k}{\vartheta_k+s}\right)^j=\frac{1}{s-\theta}\left(\frac{\hat{\psi}^-(\theta)}
{\hat{\psi}^-(s)}-1\right).
\end{split}\tag{A.12}
\end{equation}
\end{lem}

\begin{proof}[Proof of Proposition 2.1]
For given $y > b$, the function of $x$, $V_q(x)$, is defined as (see Lemma 2.5):
\begin{equation}
V_q(x) = \mathbb P_x\left(U_{e(q)} > y\right).\tag{A.13}
\end{equation}
Recall (1.1) and note that $\{X_t, t < \tau_{b}^+ \}$ and $\{U_t, t < \kappa_b^+ \}$ with $\kappa_b^+:=\inf\{t\geq 0: U_t > b \}$ under $\mathbb P_x$ have the same law if $x < b$. Thus, for $x<b$, the strong Markov property of $U$ will lead to
\begin{equation}
\begin{split}
V_q(x)
&=\mathbb E_x\left[\textbf{1}_{\{U_{e(q)} > y\}} \textbf{1}_{\{e(q) > \kappa_b^+\}}\right]=
\mathbb E_x\left[e^{-q\kappa_b^+}V_q(U_{\kappa_b^+})\right]\\
&=\mathbb E_x\left[e^{-q\tau_{b}^+}V_q(X_{\tau_{b}^+})\right]=\mathbb E\left[e^{-q\tau_{b-x}^+}V_q(X_{\tau_{b-x}^+}+x)\right]\\
&=\sum_{k=1}^{m^+}\sum_{j=1}^{m_k}C_{kj}(b-x)\int_{0}^{\infty}\frac{(\eta_k)^jz^{j-1}}{(j-1)!}e^{-\eta_k z}V_q(b+z)dz\\
&\ \ + C_0(b-x)V_q(b)=\sum_{k=1}^{M^+}J_{k}e^{\beta_{k}(x-b)}, \ \ x<b,
\end{split}\tag{A.14}
\end{equation}
where $J_1, \ldots, J_{M^+}$ are constants which are not dependent on $x$; the fourth and the fifth equality follows from Lemma A.2 (i) and Remark A.1, respectively.

For $x > b$, the strong Markov property of $U$ and the fact that $\{Y_t, t < \hat{\tau}_{b}^-\}$ under $\hat{\mathbb P}_x$ and $\{U_t, t < \kappa_b^-\}$ with  $\kappa_b^-:=\inf\{t\geq 0: U_t < b \}$  under $\mathbb P_x$ have the same law\footnote{Strictly speaking, this statement should be written as follows:
$\{Y_t, t < \tilde{\tau}_{b}^-\}$  with $\tilde{\tau}_{b}^-:= \inf\{t\geq 0: Y_t \leq b\}$ under $\hat{\mathbb P}_x$ and $\{U_t, t < \tilde{\kappa}_b^-\}$ with  ${\tilde{\kappa}_b}^-:=\inf\{t\geq 0: U_t \leq b \}$  under $\mathbb P_x$ have the same law. But, since $\sigma > 0$, we have  $\mathbb P_x\left(\hat{\tau}_{b}^-=\tilde{\tau}_{b}^-\right)=1$ and $\mathbb P_x\left(\kappa_b^-=\tilde{\kappa}_{b}^-\right)=1$.} will give
\begin{equation}
\begin{split}
&V_q(x)=\mathbb E_x\left[\int_{0}^{\kappa_b^{-}} qe^{-q t}\rm{\bf{1}}_{\{U_t > y\}}dt +
\int_{\kappa_b^{-}}^{\infty}qe^{-q t}\rm{\bf{1}}_{\{U_t > y\}}dt\right]\\
& = \hat{\mathbb E}_x\left[\int_{0}^{\infty}qe^{-q t}\rm{\bf{1}}_{\{Y_t > y, t < \hat{\tau}_{b}^{-}\}}dt\right]
+ \hat{\mathbb E}_x\left[e^{-q \hat{\tau}_{b}^{-}}V_q(Y_{\hat{\tau}_{b}^{-}})\right]\\
&=\hat{\mathbb P}_x(Y_{e(q)} > y, \underline{Y}_{e(q)} \geq b)+ \hat{\mathbb E}_x\left[e^{-q \hat{\tau}_{b}^{-}}V_q(Y_{\hat{\tau}_{b}^{-}})\right].
\end{split}\tag{A.15}
\end{equation}
Since Lemma 2.6 holds, the first item on the right-hand side of (A.15) can be written as
\begin{equation}
\begin{split}
&\int_{b-x}^{0}
\hat{\mathbb P}(Y_{e(q)}-\underline{Y}_{e(q)} > y - x -z, \underline{Y}_{e(q)}\in dz) \\
=
&\int_{b-x}^{0}\hat{\mathbb P}(\overline{Y}_{e(q)} > y - x -z)\hat{\mathbb P}(\underline{Y}_{e(q)}\in dz)\\
=
&\left\{\begin{array}{cc}
\sum_{k=1}^{\hat{M}^+}\hat{H}_{k}e^{\hat{\beta}_{k}(x-y)}+\sum_{k=1}^{\hat{N}^-}\hat{P}^{*}_{k}
e^{\hat{\gamma}_{k}(b-x)}, & b <  x \leq y,\\
1+\sum_{k=1}^{\hat{N}^-}\hat{Q}_{k}e^{\hat{\gamma}_{k}(y-x)}+\sum_{k=1}^{\hat{N}^-}\hat{P}^{*}_{k}
e^{\hat{\gamma}_{k}(b-x)}, & x \geq y,
\end{array}\right.
\end{split}\tag{A.16}
\end{equation}
where the second equality is due to Lemma A.1 (ii) and (iii) (note that $\hat{\mathbb P}\left(
\overline{Y}_{e(q)} > z\right)=1$ if $z \leq 0$); for $k=1,2,\ldots, \hat{M}^+$,
\begin{equation}
\begin{split}
\hat{H}_k=\frac{\hat{C}_k}{\hat{\beta}_k}\sum_{j=1}^{\hat{N}^-}\frac{\hat{D}_j}{\hat{\beta}_k+\hat{\gamma}_j},
\end{split}\tag{A.17}
\end{equation}
and for $k=1,2,\ldots, \hat{N}^-$,
\begin{equation}
\begin{split}
\hat{Q}_k=\hat{D}_k\sum_{i=1}^{\hat{M}^+}\frac{\hat{C}_i}{\hat{\beta_i}(\hat{\beta}_i+\hat{\gamma}_k)}-\frac{\hat{D}_k}
{\hat{\gamma}_k} \ and \
\hat{P}^{*}_k=-\sum_{i=1}^{\hat{M}^+}\frac{\hat{C}_i}{\hat{\beta}_i}
\frac{\hat{D}_k}{\hat{\beta}_i+\hat{\gamma}_k}e^{\hat{\beta}_i(b-y)}.
\end{split}\tag{A.18}
\end{equation}

Therefore, from (A.15), (A.16), Lemma A.2 (ii) and Remark A.1, we conclude that there are some constants $\hat{P}_1, \ldots, \hat{P}_{\hat{N}^-}$ (independent of $x$) such that
\begin{equation}
\begin{split}
\sum_{k=1}^{\hat{N}^-}\hat{P}_ke^{\hat{\gamma}_{k}(b-x)}
&=\sum_{k=1}^{n^-}\sum_{j=1}^{n_k}\hat{D}_{kj}(b-x)\int_{-\infty}^{0}V_q(b+z)
\frac{(\vartheta_k)^j(-z)^{j-1}}{(j-1)!}e^{\vartheta_kz}dz\\
&+\hat{D}_0(b-x)V_q(b)+\sum_{j=1}^{\hat{N}^-}\hat{P}^{*}_{j}
e^{\hat{\gamma}_{j}(b-x)}, \  \ for \ \ all \ \ x > b,
\end{split}\tag{A.19}
\end{equation}
and
\begin{equation}
V_q(x)=\left\{
\begin{array}{cc}
\sum_{k=1}^{\hat{M}^+}\hat{H}_{k}e^{\hat{\beta}_{k}(x-y)}+\sum_{k=1}^{\hat{N}^-}\hat{P}_{k}
e^{\hat{\gamma}_{k}(b-x)}, & b < x \leq y,\\
1+\sum_{k=1}^{\hat{N}^-}\hat{Q}_{k}e^{\hat{\gamma}_{k}(y-x)}+\sum_{k=1}^{\hat{N}^-}\hat{P}_{k}
e^{\hat{\gamma}_{k}(b-x)}, & x \geq y,
\end{array}\right.\tag{A.20}
\end{equation}

For the constants $J_k$ in (A.14) and $\hat{P}_k$ in (A.19), we will show in the following Lemma A.4 that formulas (A.25)--(A.28) hold.

Next, consider a rational function of $x$ as follows:
\begin{equation}
L(x)=\sum_{i=1}^{M^+}\frac{J_i}{x-\beta_i}-
\sum_{i=1}^{\hat{N}^-}\frac{\hat{P}_i}{x+\hat{\gamma}_i}-\sum_{i=1}^{\hat{M}^+}\frac{\hat{H}_i}{x-\hat{\beta}_i}
e^{\hat{\beta}_i(b-y)}. \tag{A.21}
\end{equation}
For fixed $1 \leq k \leq m^+$ and $0\leq j \leq m_k-1$, (A.27) in Lemma A.4 gives us  $\frac{\partial^j}{\partial x^j}\left(L(x)\right)_{x=\eta_k}=0$. This implies  that $\eta_k$ is a root of $L(x)=0$ and its multiplicity is $m_k$. Moreover, for $1 \leq k \leq n^-$, (A.28) means that $-\vartheta_k$ is a $n_k$-multiplicity root of $L(x)=0$. From these results, $L(x)$ can be rewritten as
\begin{equation}
\frac{\prod_{k=1}^{m^+}(x-\eta_k)^{m_k}\prod_{k}^{n^-}(x+\vartheta_k)^{n_k}(l_0+l_1x+\cdots+
l_{M^+-1}x^{M^+-1}+ x^{M^+}(L_0+L_1x)}
{\prod_{i=1}^{M^+}(x-\beta_i)\prod_{i=1}^{\hat{N}^-}(x+\hat{\gamma}_i)\prod_{i=1}^{\hat{M}^+}
(x-\hat{\beta}_i)},\tag{A.22}
\end{equation}
where $l_0,l_1, \ldots, l_{M^+-1}$, $L_0$ and $L_1$  are constants, and we have used $M^+=\sum_{k=1}^{m^+}m_k+1=\hat{M}^+$ and
$\hat{N}^-=\sum_{k=1}^{n^-}n_k +1$ (see Remark 2.5 and Lemma 2.4) in the above derivation.

Then, by applying (A.25) and (A.26), we derive $L_0=0$ and $L_1=0$ from (A.21) and (A.22). Finally, it can be seen from (A.21) that
\begin{equation}
\lim_{x\rightarrow \hat{\beta}_i}L(x)(x-\hat{\beta}_i)=-\hat{H}_ie^{\hat{\beta}_i(b-y)}, \ \ 1 \leq i \leq \hat{M}^+.\tag{A.23}
\end{equation}
Therefore, we arrive at the conclusion that
\begin{equation}
\begin{split}
&L(x)=\frac{\prod_{k=1}^{m^+}(x-\eta_k)^{m_k}\prod_{k=1}^{n^-}(x+\vartheta_k)^{n_k}}
{\prod_{i=1}^{M^+}(x-\beta_i)\prod_{i=1}^{\hat{N}^-}(x+\hat{\gamma}_i)}\times\\
&\sum_{k=1}^{\hat{M}^+}
\frac{\prod_{i=1}^{M^+}(\hat{\beta}_k-\beta_i)\prod_{i=1}^{\hat{N}^-}(\hat{\beta}_k+\hat{\gamma}_i)}
{\prod_{i=1}^{m^+}(\hat{\beta}_k-\eta_i)^{m_i}\prod_{i=1}^{n^-}(\hat{\beta}_k+\vartheta_i)^{n_i}}
\frac{-\hat{H}_k}{x-\hat{\beta}_k}e^{\hat{\beta}_k(b-y)}.
\end{split}\tag{A.24}
\end{equation}
Formulas (2.20) and (2.21) are derived from (A.14), (A.20), (A.21) and (A.24).
\end{proof}

\begin{lem}
(i) It holds that
\begin{equation}
\sum_{i=1}^{M^+}J_i=V_q(b)=\sum_{i=1}^{\hat{M}^+}\hat{H}_ie^{\hat{\beta}_i(b-y)}+\sum_{i=1}^{\hat{N}^-}\hat{P}_i, \tag{A.25}
\end{equation}
and
\begin{equation}
\sum_{i=1}^{M^+}J_i\beta_i=V_q^{\prime}(b)=\sum_{i=1}^{\hat{M}^+}\hat{H}_i\hat{\beta}_ie^{\hat{\beta}_i(b-y)}-
\sum_{i=1}^{\hat{N}^-}
\hat{P}_i\hat{\gamma}_i.
\tag{A.26}
\end{equation}

(ii) For $1 \leq k \leq m^+$ and $0\leq j \leq m_k-1$,
\begin{equation}
\sum_{i=1}^{M^+}\frac{J_i(-1)^j}{(\beta_i-\eta_k)^{j+1}}+\sum_{i=1}^{\hat{N}^-}\frac{\hat{P}_i}
{(\eta_k+\hat{\gamma}_i)^{j+1}}
-\sum_{i=1}^{\hat{M}^+}\frac{\hat{H}_i(-1)^j}{(\hat{\beta}_i-\eta_k)^{j+1}}e^{\hat{\beta}_i(b-y)}=0.\tag{A.27}
\end{equation}

(iii) For any given $1 \leq k \leq n^-$ and $0\leq j \leq n_k-1$,
\begin{equation}
\sum_{i=1}^{M^+}\frac{J_i(-1)^j}{(\beta_i+\vartheta_k)^{j+1}}+
\sum_{i=1}^{\hat{N}^-}\frac{\hat{P}_i}{(\hat{\gamma}_i-\vartheta_k)^{j+1}}
-\sum_{i=1}^{\hat{M}^+}\frac{\hat{H}_i(-1)^j}{(\hat{\beta}_i+\vartheta_k)^{j+1}}e^{\hat{\beta}_i(b-y)}=0.\tag{A.28}
\end{equation}
\end{lem}

\begin{proof}

(i) These results follow from (2.19), (A.14) and (A.20).

(ii) First, as $\sum_{i=1}^{\hat{M}^+}\frac{\hat{C}_i}{\hat{\beta}_i}=1=\sum_{j=1}^{\hat{N}^-}\frac{\hat{D}_j}{\hat{\gamma}_j}$ (let $s=0$ in (A.5) and (A.6)), for some proper $\theta$, we have (see (A.18))
\[
1+\sum_{k=1}^{\hat{N}^-}\frac{\theta \hat{Q}_k}{\theta+\hat{\gamma}_k}=
\sum_{i=1}^{\hat{M}^+}\sum_{j=1}^{\hat{N}^-}\frac{\hat{C}_i\hat{D}_j}{\hat{\beta}_i
\hat{\gamma}_j}
+\sum_{k=1}^{\hat{N}^-}\frac{\theta}{\theta+\hat{\gamma}_k}\left(
\sum_{i=1}^{\hat{M}^+}\frac{\hat{D}_k\hat{C}_i}{\hat{\beta_i}(\hat{\beta}_i+\hat{\gamma}_k)}-\frac{\hat{D}_k}
{\hat{\gamma}_k} \sum_{i=1}^{\hat{M}^+}\frac{\hat{C}_i}{\hat{\beta}_i}\right).
\]
Thus, for all $\theta \in \mathbb C $ except at $\hat{\beta}_1, \ldots, \hat{\beta}_{\hat{M}^+}$ and $-\hat{\gamma}_1, \ldots, -\hat{\gamma}_{\hat{N}^-}$, the last formula and (A.17) will lead to
\begin{equation}
\begin{split}
&\sum_{k=1}^{\hat{M}^+}\frac{\theta \hat{H}_k}{\hat{\beta}_k-\theta}+\sum_{k=1}^{\hat{N}^-}\frac{\theta \hat{Q}_k}{\theta+\hat{\gamma}_k}+1\\
&=\sum_{i=1}^{\hat{M}^+}\sum_{j=1}^{\hat{N}^-}\left\{\frac{\theta \hat{C}_i\hat{D}_j}{\hat{\beta}_i(\hat{\beta}_i+\hat{\gamma}_j)
(\hat{\beta}_i-\theta)}-\frac{\theta\hat{C}_i\hat{D}_j}{\hat{\gamma}_j(\hat{\beta}_i+\hat{\gamma}_j)
(\hat{\gamma}_j+\theta)}+\frac{\hat{C}_i\hat{D}_j}{\hat{\beta}_i\hat{\gamma}_j}\right\}\\
&=\sum_{i=1}^{\hat{M}^+}\sum_{j=1}^{\hat{N}^-}\frac{\hat{C}_i\hat{D}_j}{(\hat{\beta}_i-\theta)
(\theta+\hat{\gamma}_j)}=\hat{\psi}^+(-\theta)\hat{\psi}^-(\theta),
\end{split}\tag{A.29}
\end{equation}
where the last equality follows from (A.5) and (A.6).

Note that
\begin{equation}
\frac{\partial^{j-1}}{\partial \eta^{j-1}}\left(\hat{\psi}^+(-\eta)\right)_{\eta=\eta_k}=0, \ \ for \ \ 1\leq k \leq m^+ \ \ and \ \ 1\leq j\leq m_k.\tag{A.30}
\end{equation}
From (A.29) and (A.30), we obtain
\begin{equation}
\frac{(\eta_k)^j(-1)^{j-1}}{(j-1)!}\frac{\partial^{j-1}}{\partial \eta^{j-1}}\left(
\frac{1}{\eta}e^{\eta(b-y)}\Big(\sum_{i=1}^{\hat{M}^+}\frac{\hat{H}_i\eta}{\hat{\beta}_i-\eta}+
\sum_{i=1}^{\hat{N}^-}\frac{\eta \hat{Q}_i}{\eta+\hat{\gamma}_i}+1\Big)\right)_{\eta=\eta_k}=0.\tag{A.31}
\end{equation}

For $1\leq k \leq m^+$ and $1\leq j \leq m_k$, the integral $(-1)^{j-1}\int_{z_1}^{z_2}z^{j-1}e^{-\eta_k z} e^{ \xi z} dz$ can be understood as
$\frac{\partial^{j-1}}{\partial \eta^{j-1}}\left(\int_{z_1}^{z_2}e^{-\eta z} e^{\xi z} dz\right)_{\eta=\eta_k}$ for some proper constants $z_1$, $z_2$ and $\xi$, then it can be obtained from (A.20) and (A.31) that
\[
\int_{0}^{\infty}\frac{(\eta_k)^jz^{j-1}}{(j-1)!}e^{-\eta_k z}V_q(b+z)dz=
\sum_{i=1}^{\hat{N}^-}\frac{\hat{P}_i(\eta_k)^j}{(\eta_k+\hat{\gamma}_i)^j}+\sum_{i=1}^{\hat{M}^+}\frac{
\hat{H}_i(\eta_k)^j}{(\eta_k-\hat{\beta}_i)^j}e^{\hat{\beta}_i(b-y)},
\]
which combined with (A.14) and the result of $V_q(b)=\sum_{i=1}^{\hat{M}^+}\hat{H}_ie^{\hat{\beta}_i(b-y)}+\sum_{i=1}^{\hat{N}^-}\hat{P}_i$ (see (A.25)), gives
\begin{equation}
\begin{split}
\sum_{k=1}^{M^+}J_{k}e^{\beta_{k}(x-b)}
=
&\sum_{k=1}^{m^+}\sum_{j=1}^{m_k}C_{kj}(b-x)\left(
\sum_{i=1}^{\hat{N}^-}\frac{\hat{P}_i(\eta_k)^j}{(\eta_k+\hat{\gamma}_i)^j}+\sum_{i=1}^{\hat{M}^+}\frac{
\hat{H}_i(\eta_k)^j}{(\eta_k-\hat{\beta}_i)^j}e^{\hat{\beta}_i(b-y)}\right)\\ &+C_0(b-x)\left(\sum_{i=1}^{\hat{M}^+}\hat{H}_ie^{\hat{\beta}_i(b-y)}+\sum_{i=1}^{\hat{N}^-}\hat{P}_i\right),
\ \ for \ \ all \ \ x<b.
\end{split}\tag{A.32}
\end{equation}

It follows from (A.11) and (A.32) that
\begin{equation}
\begin{split}
&\sum_{i=1}^{M^+}\frac{J_i}{\beta_i+\theta}=\int_{-\infty}^{b}e^{\theta(x-b)}\sum_{i=1}^{M^+}J_ie^{\beta_i(x-b)}dx\\
&=\sum_{i=1}^{\hat{M}^+}\frac{\hat{H}_ie^{\hat{\beta}_i(b-y)}}{\theta+\hat{\beta}_i}\left(1-\frac{\psi^+(\theta)}
{\psi^+(-\hat{\beta}_i)}\right)+\sum_{i=1}^{\hat{N}^-}\frac{\hat{P}_i}{\hat{\gamma}_i-\theta}\left(\frac{\psi^+(\theta)}
{\psi^+(\hat{\gamma}_i)}-1\right).
\end{split}\tag{A.33}
\end{equation}
Since $\lim_{\theta \rightarrow -\hat{\beta}_i} \frac{\psi^+(-\hat{\beta}_i)-\psi^+(\theta)}{\theta+\hat{\beta}_i}=-\psi^{+\prime}(-\hat{\beta}_i)$ and $\lim_{\theta \rightarrow \hat{\gamma}_i} \frac{\psi^+(\theta)-\psi^+(\hat{\gamma}_i)}{\theta -\hat{\gamma}_i}=\psi^{+\prime}(\hat{\gamma}_i)$.
In addition, noting that both sides of (A.33) are rational functions\footnote{Here, we omit the first equality in (A.33), i.e., the item $\int_{-\infty}^{b}e^{\theta(x-b)}\sum_{i=1}^{M^+}J_ie^{\beta_i(x-b)}dx$ is omitted.} of $\theta$, we can extend identity (A.33) to the whole plane except at $-\beta_1, \ldots, -\beta_{M^+}$. Then, for given $1 \leq k \leq m^+$ and $0\leq j \leq m_k-1$,  (A.27) is derived  by first taking  a derivative on both sides of (A.33) with respect to $\theta $ up to $j$ order and then letting $\theta$ equal to $-\eta_k$,  where we have used the fact that $\frac{\partial ^{j}}{\partial \theta^{j}}\Big(\psi^+(\theta)\Big)_{\theta=-\eta_k}=0$.

(iii) Similarly, for $1\leq k \leq n^-$ and $1\leq j \leq n_k$, it follows from (A.14) that
\begin{equation}
\int_{-\infty}^{0}V_q(b+z)\frac{(\vartheta_k)^j(-z)^{j-1}}{(j-1)!}e^{\vartheta_k z}dz=\sum_{i=1}^{M^+}\frac{J_i(\vartheta_k)^j}{(\vartheta_k+\beta_i)^j}.\tag{A.34}
\end{equation}
From (A.12), (A.18), (A.19) and (A.34) and the fact of $V_q(b)=\sum_{i=1}^{M^+}J_i$ (see (A.25)), it can be proved that
\begin{equation}
\begin{split}
&\sum_{i=1}^{\hat{N}^-}\frac{\hat{P}_i}{\theta+\hat{\gamma}_i}=\sum_{i=1}^{\hat{N}^-}\hat{P}_i\int_{b}^{\infty}e^{\theta(b-x)}
e^{\hat{\gamma}_i(b-x)}dx\\
&=-\sum_{i=1}^{\hat{M}^+}\frac{\hat{C}_i}{\hat{\beta}_i}e^{\hat{\beta}_i(b-y)}
\sum_{j=1}^{\hat{N}^-}\frac{\hat{D}_j}{\hat{\beta}_i+\hat{\gamma}_j}\frac{1}{\theta+\hat{\gamma}_j}+
\sum_{i=1}^{M^+}\frac{J_i}{\beta_i-\theta}\left(\frac{\hat{\psi}^-(\theta)}{\hat{\psi}^-(\beta_i)}-1\right).
\end{split}\tag{A.35}
\end{equation}
In addition, we note that
\begin{equation}
\begin{split}
&-\sum_{i=1}^{\hat{M}^+}\frac{\hat{C}_i}{\hat{\beta}_i}e^{\hat{\beta}_i(b-y)}
\sum_{j=1}^{\hat{N}^-}\frac{\hat{D}_j}{\hat{\beta}_i+\hat{\gamma}_j}\frac{1}{\theta+\hat{\gamma}_j}\\
&=-\sum_{i=1}^{\hat{M}^+}\frac{\hat{C}_i}{\hat{\beta}_i}e^{\hat{\beta}_i(b-y)}\frac{1}{\hat{\beta}_i-\theta}
\sum_{j=1}^{\hat{N}^-}\hat{D}_j\left(\frac{1}{\theta+\hat{\gamma}_j}-\frac{1}{\hat{\beta}_i+\hat{\gamma}_j}\right)\\
&=\sum_{i=1}^{\hat{M}^+}\frac{\hat{H}_i}{\hat{\beta}_i-\theta}e^{\hat{\beta}_i(b-y)}-\sum_{i=1}^{\hat{M}^+}
\frac{\hat{C}_i\hat{\psi}^-(\theta)}{\hat{\beta}_i(\hat{\beta}_i-\theta)}e^{\hat{\beta}_i(b-y)}.
\end{split}\tag{A.36}
\end{equation}
where the second equality follows from (A.6) and (A.17).

Hence, from the last two formulas, we arrive at
\begin{equation}
\begin{split}
&\sum_{i=1}^{\hat{N}^-}\frac{\hat{P}_i}{\theta+\hat{\gamma}_i}=
\sum_{i=1}^{M^+}\frac{J_i}{\beta_i-\theta}\left(\frac{\hat{\psi}^-(\theta)}{\hat{\psi}^-(\beta_i)}-1\right)\\
&+\sum_{i=1}^{\hat{M}^+}\frac{\hat{H}_i}{\hat{\beta}_i-\theta}e^{\hat{\beta}_i(b-y)}-\sum_{i=1}^{\hat{M}^+}
\frac{\hat{C}_i\hat{\psi}^-(\theta)}{\hat{\beta}_i(\hat{\beta}_i-\theta)}e^{\hat{\beta}_i(b-y)},
\end{split}\tag{A.37}
\end{equation}
which holds for $\theta \in \mathbb C$ except at $-\hat{\gamma}_1, \ldots, -\hat{\gamma}_{\hat{N}^-}$.

For given $1 \leq k \leq n^-$, on both sides of (A.37), we take a derivative with respect to $\theta $ up to $j$ order for $0\leq j \leq n_k-1$ and then let $\theta$ equal to $-\vartheta_k$.
This calculation leads to (A.28) since $\frac{\partial ^{j}}{\partial \theta^{j}}\Big(\hat{\psi}^-(\theta)\Big)_{\theta=-\vartheta_k}=0$, and the proof is completed.
\end{proof}

\end{appendix}

\end{document}